\documentclass{amsart}
\usepackage{latexsym}
\usepackage{amsmath}
\usepackage{amssymb}
\usepackage[all]{xy}

\newtheorem{thm}{Theorem}[section]
\newtheorem{prop}[thm]{Proposition}
\newtheorem{cor}[thm]{Corollary}
\newtheorem{lem}[thm]{Lemma}
\theoremstyle{definition}
\newtheorem{defn}[thm]{Definition}

\theoremstyle{remark}
\newtheorem{rem}[thm]{Remark}
\newtheorem{ex}[thm]{Example}
\newtheorem{claim}[thm]{Claim}

\newcommand{\aut}[1]{\text{\rm aut}_1({#1})}
\newcommand{\K}{{\mathbb Q}}
\newcommand{\F}{{\mathcal F}}
\newcommand{\mapright}[1]{%
 \smash{\mathop{%
  \hbox to 1cm{\rightarrowfill}}\limits_{#1}}}
\newcommand{\maprightd}[2]{%
 \smash{\mathop{%
  \hbox to 1.2cm{\rightarrowfill}}\limits^{#1}\limits_{#2}}}
\newcommand{\mapleft}[1]{%
 \smash{\mathop{%
  \hbox to 1cm{\leftarrowfill}}\limits_{#1}}}
\newcommand{\mapleftu}[1]{%
 \smash{\mathop{%
  \hbox to 0.8cm{\leftarrowfill}}\limits^{#1}}}
\newcommand{\maprightu}[1]{%
 \smash{\mathop{%
  \hbox to 1cm{\rightarrowfill}}\limits^{#1}}}
\newcommand{\maprightud}[2]{%
 \smash{\mathop{%
  \hbox to 1cm{\rightarrowfill}}\limits^{#1}_{#2}}}
\newcommand{\mapleftud}[2]{%
 \smash{\mathop{%
  \hbox to 1cm{\leftarrowfill}}\limits^{#1}_{#2}}}

\newcounter{eqn}[section]

\def\theeqn{\textnormal{(\thesection.\arabic{eqn})}}

\def\eqnlabel#1{%
  \refstepcounter{eqn}%
  \label{#1}%
  \leqno{\theeqn}}
\begin{document}

\title[]
{On extensions of a symplectic class}

\footnote[0]{ 

This research was partially supported by a Grant-in-Aid for Scientific
Research (C)20540070
from Japan Society for the Promotion of Science.

Department of Mathematical Sciences, 
Faculty of Science,  
Shinshu University,   
Matsumoto, Nagano 390-8621, Japan   
e-mail:{\tt kuri@math.shinshu-u.ac.jp}

}

\author{Katsuhiko KURIBAYASHI}
\date{}
   
\maketitle

\begin{abstract}
Let ${\mathcal F}$ be a fibration on a simply-connected base with
 symplectic fibre $(M, \omega)$. Assume that the fibre is nilpotent and 
${\mathbb T}^{2k}$-{\it separable}
 for some integer $k$ or a nilmanifold. Then our main theorem, Theorem \ref{thm:kappa}, 
gives a necessary and sufficient condition 
for the cohomology class $[\omega]$ to extend to a cohomology class of the
 total space of ${\mathcal F}$. This allows us to describe 
Thurston's criterion 
for a symplectic fibration to admit a compatible symplectic form 
in terms of the classifying map for the underlying fibration.    
The obstruction due to Lalond and McDuff 
for a symplectic bundle to be Hamiltonian is also rephrased in the same vein.  
Furthermore, with the aid of the main theorem, 
we discuss a global nature of the set of the homotopy equivalence classes of fibrations with symplectic fibre 
in which the class $[\omega]$ is extendable.
\end{abstract}

\section{Introduction}

Let $P \to B$ be a fibration over a simply-connected space with fibre equal to 
a closed symplectic manifold $(M, \omega)$. 
We consider the problem whether the cohomology class $[\omega]$ in  
$H^2(M; {\mathbb R})$ extends to a cohomology class of $P$ 
when the fibre $(M, \omega)$ is in a large family 
of symplectic manifolds, which contains the product spaces   
of simply-connected symplectic manifolds and the even dimensional torus.

Let $\text{Symp}(M, \omega)$ be 
the group of symplectomorphisms of $(M, \omega)$, namely diffeomorphisms 
of $M$ that fix the symplectic form $\omega \in \Omega_{\text{de Rham}}^2(M)$. 
A locally trivial fibration with symplectic fibre $(M, \omega)$ in the category of smooth manifolds is
called a {\it symplectic fibration} if its structural group is contained in 
$\text{Symp}(M, \omega)$. 
The following result due to Thurston tells us importance of 
such an extension of a symplectic class.  

\begin{thm}     
\cite[Theorem 6.3]{M-S} \cite{T}
\label{thm:M-S}
Let $\pi : P \to B$ be a compact symplectic fibration with symplectic
 fibre $(M, \omega)$ and connected symplectic base $(B, \beta)$. 
 Denote by $\omega_b$ the canonical symplectic form on the fibre $P_b$
 and suppose that there is a class $a \in H^2(P; {\mathbb R})$ such
 that $i_b^*a = [\omega_b]$ for some $b \in B$. Then for every sufficient
 large real number $K > 0$, there exists a symplectic form 
$\omega_K$ of $P$ such that $\omega_b= i_b^* \omega_K$ and 
$a + K[\pi^*\beta]= [\omega_K]$.   
\end{thm} 

Another significant result, which motivates us to investigate the extension 
of a symplectic class, is related to a reduction problem of the
structural group of a bundle. 
To describe the result,  we recall a subgroup of $\text{Symp}(M, \omega)$.    
A smooth map $\phi \in \text{Symp}(M, \omega)$ is called  
a {\it Hamiltonian symplectomorphism} if $\phi$ is the time $1$-map $\phi_1$ 
of the Hamiltonian isotopy $\phi_t$, $t \in [0, 1]$; that is, $\phi_0 = id_M$, 
$\phi_t \in \text{Symp}(M, \omega)$ for any $t$ and 
$\omega (X_t, \cdot) = dH_t$ with 
$\frac{d}{dt}\phi_t = X_t \circ \phi_t$   
for some time dependent function $H_t : M \to {\mathbb R}$. 
We denote by $\text{Ham}(M, \omega)$ the subgroup of $\text{Symp}(M, \omega)$ 
consisting of Hamiltonian symplectomorphisms. 
%symplectomorphisms is a normal subgroup of $\text{Symp}(M, \omega)$; 
%for example see \cite[Proposition 10.2]{M-S}. 

Let ${\mathcal F} : (M, \omega) \to P \to B$ be a fibre bundle  which is not 
necessarily in the category of smooth manifolds. 
We shall say that the bundle ${\mathcal F}$ is {\it symplectic}  
if its structural group is the group of symplectomorphisms $\text{Symp}(M, \omega)$.   
A symplectic  bundle $(M, \omega) \to P \to B$ is said to be 
{\it Hamiltonian} if the structural group may be reduced to 
the subgroup $\text{Ham}(M, \omega)$. 

The following result due to Lalonde and McDuff gives  
a characterization of Hamiltonian bundles. 

\begin{thm} \cite[Lemma 2.3]{L-M}
\label{thm:L-M} 
A symplectic bundle $(M, \omega) \to P \to B$ over a simply-connected
 space $B$ is Hamiltonian if and only if the class 
$[\omega] \in H^2(M; {\mathbb R})$ extends to a class of $H^2(P; {\mathbb R})$. 
\end{thm}

We recall here the key consideration in the proof of 
\cite[Proposition 3.1]{K-M}.

\begin{rem}
\label{rem:coupling}
Let $\{E_r, d_r\}$ be the Leray-Serre spectral sequence of a fibration $P \to B$ 
over simply-connected base $B$. 
Suppose that the fibre is a  
$2m$-dimensional symplectic manifold $(M, \omega)$.  
Assume further that $d_2([\omega])=0$ in the $E_2$-term 
for $[\omega] \in E_2^{0,2}=H^2(M; {\mathbb R})$. 
Then $[\omega]$ is a permanent cycle. In fact, we have 
$$
0 = d_3([\omega]^{m+1})=(m+1)[\omega]^m\otimes d_3([\omega]) \in E_3^{2m, 3} 
\subset H^{2m}(M) \otimes H^3(B). 
$$
This implies that $d_3([\omega]) = 0$. 
\end{rem}

The argument in Remark \ref{rem:coupling} enables us to deduce 
the following proposition. 

\begin{prop}
 \label{prop:simply-connected}
Let $(M, \omega) \to P \to B$ be a fibration as 
in Remark \ref{rem:coupling}. If $H^1(M; {\mathbb R}) = 0$, then   
the symplectic class $[\omega]$ extends to  a cohomology class of $P$.
\end{prop}

Thus one might take an interest in the extension problem of 
a symplectic class in the case where $H^1(M; {\mathbb R}) \neq 0$ 
for the given fibre $M$. In what follows, for a fibration  $(M, \omega) \to P \to B$, 
we may call the class $[\omega]$ {\it extendable} in the fibration if 
it extends to a cohomology class of the total space $P$. 

We deal with such an extension problem 
assuming that the even dimensional torus is rationally separable from
the fibre.   
In order to explain the separability more precisely, 
we recall some terminology
from rational homotopy theory. We refer the reader to \cite{B-G} and \cite{F-H-T}  for more details.

We denote by $\wedge V$ the free graded algebra generated by 
a graded vector space $V$.  
Let $A_{PL}(X)$ be the differential graded algebra (DGA) 
of polynomial forms on a space $X$. We shall say that a DGA $(\wedge V, d)$ is
a {\it model} for $X$ if there exits a quasi-isomorphism 
$q : (\wedge V, d) \to A_{PL}(X)$, namely an isomorphism on the cohomology. 
Moreover the model is called {\it minimal} if $d(v)$ is decomposable in $\wedge V$ 
for any $v \in V$. 
%Here $\wedge V$ denotes the free graded algebra generated by a graded vector
%space $V$ over $\K$.  

\begin{defn} 
\label{defn:separable}
A closed symplectic manifold $(M, \omega)$ is 
${\mathbb T}^{2k}$-{\it separable} if it  admits a
 minimal model $(\wedge V, d)$ of the form 
$$
(\wedge V, d)=\bigotimes_{i=1}^k(\wedge (t_{i_1}, t_{i_2})), 0)
\otimes (\wedge Z, d)
$$ 
for which $(Z)^1= 0$ and there exists a cycle $\beta \in Z^2$ such that in $H^2(M; {\mathbb R})$ 
 $$[\omega] = q \gamma([\beta]) 
+ \sum_{i=1}^kq_i\gamma ([t_{i_1}t_{i_2}])$$ 
for some non-zero real numbers $q$ and $q_i$. 
Here $\gamma : H^*(M; \K) \to H^*(M; {\mathbb R})$ denotes the map induced
 by the inclusion  $\K \to {\mathbb R}$. 
\end{defn}

\begin{ex}
\label{ex:T-separable} 
Let $({\mathbb T}^{2k}, \beta')$ be the $2k$-dimensional 
torus with the standard symplectic structure $\beta'$ 
and let $({\mathbb T}^{2k}, \beta') 
\stackrel{i}{\to} P \stackrel{\pi}{\to} (M, \beta^M)$  be a Hamiltonian
 bundle over a simply-connected symplectic manifold $(M, \beta^M)$. 
Suppose that $\beta^M$ is rational in the sense that the class
 $[\beta^M]$ 
is in the image of $\gamma : H^2(M; \K) \to H^2(M; {\mathbb R})$ up to
the multiplication by a non-zero scalar.  Then the total space $P$ is a 
${\mathbb T}^{2k}$-separable symplectic manifold with an appropriate
 symplectic structure. 
 
To see this, let $(\wedge V, d)$ be a minimal model for $M$. Since $\beta^M$ is
 rational, we can choose a cycle $\beta \in V^2$ so that 
$\gamma [\beta] = q[\beta^M]$ for some non-zero real number $q$. It
 follows from the proof of \cite[Theorem 1.2]{St} that 
$P$ admits a Sullivan model of the form 
$(\wedge V, d)\otimes \bigotimes_{i=1}^k(\wedge (t_{i_1}, t_{i_2}), 0)$
in which $[\beta'] = i^*\gamma (\sum [t_{i_1}t_{i_2}])$. By virtue of 
 Theorem \ref{thm:M-S}, we see that,  
for every sufficient large real number $K$, 
there exists a symplectic form $\omega_K$ on $P$ such that 
$$
[\omega_K] = \sum_{i=1}^k \gamma  [t_{i_1}t_{i_2}] + K\pi^* [\beta^M] = 
 \sum_{i=1}^k \gamma  [t_{i_1}t_{i_2}] + Kq^{-1}\gamma [\beta].  
$$  
By definition, the symplectic manifold $(P, \omega_K)$ is 
${\mathbb T}^{2k}$-separable. 
Moreover we see that  
$P$ is a nilpotent space. This follows from 
the naturality of the action of the 
fundamental group on the higher homotopy groups. 
\end{ex}

Let $(M, \omega) \to P \to B$ be a fibration over a simply-connected base 
with symplectic fibre. 
The purpose of this paper is to exhibit a necessary and sufficient condition 
for the symplectic class $[\omega]$ to extend to a cohomology class of $P$ 
provided $(M, \omega)$ is nilpotent ${\mathbb T}^{2k}$-separable for some
integer $k$. 

Unless otherwise explicitly stated, we assume that a space is well-based and 
has the homotopy type of a connected CW complex with rational cohomology of finite type. 
For a nilpotent space $X$, we denote by  $X_\K$ the rationalization of the space $X$. 
We shall write $H^*(X)$ for the cohomology of a space $X$ with coefficients in
$\K$. 

Let $(M, \omega)$ be a nilpotent ${\mathbb T}^{2k}$-separable symplectic
manifold which admits the minimal model described in Definition \ref{defn:separable}. 
Then we can choose a map $p_M : M_\K \to {\mathbb T}^{2k}_\K$ so that 
$p_M^*(s_{i_\lambda})= t_{i_\lambda}$ for appropriate generators 
 $s_{1_\lambda}, ..., s_{k_\lambda}$ ($\lambda = 1, 2$) 
of $H^*({\mathbb T}^{2k})$, where 
$p_M^* : H^*({\mathbb T}^{2k})= H^*({\mathbb T}^{2k}_\K) \to H^*(M_\K)=
H^*(M)$ denotes the map induced by $p_M$.       
In what follows, a ${\mathbb T}^{2k}$-separable symplectic manifold $(M, \omega)$ is considered 
one equipped with such a map $p_M: M_\K \to {\mathbb T}^{2k}_\K$. 

Let $\text{aut}(M)$ be the monoid of self-homotopy equivalences of
a space $M$ and $\aut{M}$ its identity component.  We denote by $B{\mathcal G}$ 
the classifying space of a monoid ${\mathcal G}$ with identity. 
When considering the extension problem of a symplectic class in our setting, 
a linear map form $H^1({\mathbb T}^{2k})$ to $H^2(B\aut{M})$ plays an important role. 

\begin{defn}
\label{defn:kappa}
Let  $(M, \omega)$ be a ${\mathbb T}^{2k}$-separable symplectic manifold.   
The {\it detective map} $\kappa : H^1({\mathbb T}^{2k}) \to 
H^2(B\aut{M})$ of $M$ is defined to be the composite 
$$
\xymatrix@C15pt{
 H^1({\mathbb T}^{2k}) \ar[r]^{p_M^*} & H^1(M) \ar[r]^(0.4){ev^*} &   
H^1(\aut{M}) \ar[r]^(0.45){\tau} & H^2(B\aut{M}), 
}
$$
where $ev : \aut{M} \to M$ is the evaluation map at the basepoint of $M$ and 
the map $\tau : H^1(\aut{M}) \to H^2(B\aut{M})$ is the transgression of
the Leray-Serre spectral sequence of the universal fibration 
$\aut{M}\to E{\aut{M}} \to B\aut{M}$. 
\end{defn}

To describe our main theorem and its applications, 
we moreover recall from \cite{May1} some terminology and results. 

For a given space $M$, 
let $M{\mathcal W}$ be the {\it category of fibres} described in \cite[Example 6.6 (ii)]{May1}; that is,  
$X \in M{\mathcal W}$ if $X$ is of the same homotopy type of $M$ and the morphisms 
in $M{\mathcal W}$  are homotopy equivalences. 
We shall say that a map $E \to B$ is an $M{\mathcal W}$-{\it fibration} over $B$ if it is a fibration with fibre in $M{\mathcal W}$. 
Let $\pi: E \to B$ and $\nu: E' \to B'$ be $M{\mathcal W}$-fibrations. 
An $M{\mathcal W}$-{\it map} $(f, g) : \pi \to \nu$ is a pair of maps $f : E \to E'$ and 
$g : B \to B'$ such that  
$\nu\circ f = g \circ \pi$ and $f :\pi^{-1}(b) \to \nu^{-1}(g(b))$ is in $M{\mathcal W}$ for each $b \in B$. 
For any $M{\mathcal W}$-fibrations $\pi$ and $\nu$ over $B$, we write $f: \pi \to \nu$ for 
$(f, 1_B) : \pi \to \nu$. 

Let $\pi: E \to B$ and $\nu: E' \to B$ be $M{\mathcal W}$-fibrations over $B$. 
By definition a homotopy over $B$ from $g : \pi \to \nu$ to $g ':  \pi \to \nu$ is a $M{\mathcal W}$-map 
$(H, h) : \pi\times 1_I  \to \nu$; that is, it fits into the commutative diagram 
$$
\xymatrix@C30pt@R15pt{
E\times I \ar[r]^H \ar[d]_{\pi\times 1_I} & E' \ar[d]^{\nu} \\
B\times I \ar[r]_h & B  
}
$$
in which $H(x, 0)=g(x)$, $H(x, 1)=g'(x)$ for $x\in E$ and 
$h(b, t) = b$ for $(b, t) \in B\times I$. We write $g \simeq g'$ if there exists a homotopy over $B$ form $g$ to $g'$. 
We shall say that $M{\mathcal W}$-fibrations $\pi$ and $\pi'$ are homotopy equivalent if 
there exist maps $f: \pi \to \pi'$ and $f' : \pi' \to \pi$ such that $f'f \simeq 1$ and $ff' \simeq 1$. 

Let ${\mathcal E}M{\mathcal W}(B)$ be the set of homotopy equivalence classes of $M{\mathcal W}$-fibrations over $B$.  
We recall the universal cover $\pi_0(\text{aut}(M)) \to B\aut{M} \to 
B\text{aut}(M)$. 
Assume that $B$ is simply-connected.  
Then every map $B \to B\text{aut}{M}$ 
factors through $B\aut{M}$.  
Thus the result \cite[Theorem 9.2]{May1} allows us to conclude that the map $\Psi$, which sends 
a map $f : B \to B\aut{M}$ to the pullback of the universal $M$-fibration  
$M \to M_{\aut{M}} \to B\aut{M}$ by $f$, gives rise to a natural isomorphism   
$$
\Psi : [B, B\aut{M}] \stackrel{\cong}{\longrightarrow}  {\mathcal E}M{\mathcal W}(B). 
$$
We also refer the reader to \cite{Stasheff} for the classifying theorem of fibrations.  
Let ${\mathcal F}$ be an $M{\mathcal W}$-fibration. 
As usual, we call a representative $f : B \to B\aut{M}$ of $\Psi^{-1}([{\mathcal F}])$ 
the {\it classifying map} for ${\mathcal F}$.

We are now ready to describe our main theorem.

\begin{thm} 
\label{thm:kappa}
Let $(M, \omega)$ be a nilpotent ${\mathbb T}^{2k}$-separable symplectic manifold and 
${\mathcal F} : (M, \omega)  \to P \stackrel{}{\to}  B$ a  fibration over a simply-connected space $B$. 
Let $f : B \to B\aut{M}$ be the classifying map for ${\mathcal F}$. 
Then the symplectic class $[\omega] \in H^2(M; {\mathbb R})$ extends to 
a cohomology class of $P$ 
if and only if the composite 
$$H^*(f)\circ \kappa : H^1({\mathbb T}^{2k}) \to H^2(B\aut{M}) \to H^2(B)
$$ 
is trivial, where $\kappa$ is the detective map of $M$. 
\end{thm}

The novelty here is that we can describe a criterion
for the given symplectic class to extend to a cohomology class of the total space 
in terms of the detective map and the
classifying map for the fibration. The advantage of the criterion is illustrated below 
with many applications.

The ${\mathbb T}^{2k}$-separable manifold $P$ constructed in 
Example \ref{ex:T-separable} may be trivial, namely the product space
of $M$ and  ${\mathbb T}^{2k}$. 
However the structure of the Sullivan model for the monoid $\aut{P}$ is in general complicated even in that case.   
On the other hand, we turn this fact to our advantage. 
Indeed, the complexity and Theorem \ref{thm:kappa} enable us to obtain many essentially different 
non-trivial fibrations  with symplectic fibre $(M, \omega)$ in which the class $[\omega]$ is extendable; 
see Corollary  \ref{cor:more_ex3} and Example \ref{ex:Rationalized fibration} below. 
We next deal with such a global nature of fibrations.

Let $B$ be a simply-connected co-H-space with the comultiplication $\Delta : B \to B\vee B$. 
Then the homotopy set of based maps $[B, B\aut{M}]_*$ has a product defined by 
$$
[f]\ast [g]= [\nabla \circ f\vee g \circ \Delta]
$$ for $[f]$ and  $[g] \in [B, B\aut{M}]_*$, where $\nabla : B\aut{M}\vee B\aut{M} \to B\aut{M}$ is  
the folding map. 
Since $B\aut{M}$ is simply-connected, it follows that the natural map $\theta$ from the homotopy set of based maps 
$[B, B\aut{M}]_*$ to the the homotopy set $[B, B\aut{M}]$ is bijective. Thus the product on $[B, B\aut{M}]_*$ gives rise to 
that on $[B, B\aut{M}]$. In consequence, ${\mathcal E}M{\mathcal W}(B)$ has a product via the bijection $\Psi$ mentioned above. 
Observe that the product $\ast$ on ${\mathcal E}M{\mathcal W}(B)$ is represented in terms of fibrations. 
In fact,  let $f, f' : B \to \aut{M}$ be 
the classifying maps for fibrations ${\mathcal F}$ and ${\mathcal F}'$, respectively. Since $\theta$ is bijection, without loss of generality, 
we can assume that $f$ and $f'$ are based maps. 
Let ${\mathcal F}_{f\ast f'}$ be the bullback of the universal $M$-fibration by
$f\ast f' := \nabla \circ f\vee f'\circ\Delta$. Then we see that in ${\mathcal E}M{\mathcal W}(B)$
$$
[{\mathcal F}]\ast [{\mathcal F}'] = [{\mathcal F}_{f\ast f'}]. 
$$

We define ${\mathcal M}_{(M, \omega)}^{ex}(B)$ to be the subset of ${\mathcal E}M{\mathcal W}(B)$  consisting of  
classes $[{\mathcal F}]$ such that the cohomology class $[\omega]$ is extendable in some representative of the form 
$M\to E \to B$ of $[{\mathcal F}]$ and hence in any representative with fibre $M$. 
By virtue of Theorem \ref{thm:kappa}, we establish the following result. 

\begin{prop}\label{prop:more_ex1}  Let $(M, \omega)$ be a nilpotent ${\mathbb T}^{2k}$-separable symplectic manifold and let 
$B$ be a simply-connected co-H-space. Then 
${\mathcal M}_{(M, \omega)}^{ex}(B)$ is closed under the product on ${\mathcal E}M{\mathcal W}(B)$ mentioned above. 
\end{prop}

Let $B$ be the double suspension of a space $B'$ which is not necessarily connected. 
As usual,  we consider $B=S^2\vee B'$ a co-H-space whose comultiplication is defined by the standard that of 
$S^2$. 
Then $[B, B\aut{M}]_*$ is an abelian group 
with respect to the product mentioned above and hence so is ${\mathcal E}M{\mathcal W}(B)$.  
Suppose that $(M, \omega)$ is a nilpotent ${\mathbb T}^{2k}$-separable symplectic manifold. 
Proposition \ref{prop:more_ex1} implies that 
${\mathcal M}_{(M, \omega)}^{ex}(B)$ is also abelian.   
Therefore the dimension of 
the subspace ${\mathcal M}_{(M, \omega)}^{ex}(B)\otimes \K$ of the vector space ${\mathcal E}M{\mathcal W}(B)\otimes \K$ 
is in our great interest. 

Let $W$ be the vector space  $QH^*(B\aut{M})$ of indecomposable elements of $H^*(B\aut{M})$. 
Since $B\aut{M}$  is simply-connected, 
it follows that $W$ contains $H^2(B\aut{M})$ and in particular the image of the detective map $\kappa$ of $M$. 
Suppose that $B$ is the sphere $S^2$. 
Theorem \ref{thm:kappa} enables us to determine explicitly 
the dimension of ${\mathcal M}_{(M, \omega)}^{ex}(S^2)\otimes \K$ with $\dim W^2$. 

\begin{prop}\label{prop:more_ex2}  Let $B$ be the double suspension of a finite CW complex and $M$ a nilpotent space. 
Suppose that $H^*(B\aut{M})$ is a free algebra or $B$ is the sphere $S^2$. Then 
$${\mathcal E}M{\mathcal W}(B)\otimes \K\ \cong  \text{\em Hom}_{GV}(W, H^*(B))$$ 
as a vector space, where $W = QH^*(B\aut{M})$ and 
$\text{\em Hom}_{GV}(W, H^*(B))$ denotes the vector space of linear maps 
from $W$ to $H^*(B)$ of degree zero. 
Assume further that $(M, \omega)$ is a nilpotent ${\mathbb T}^{2k}$-separable symplectic manifold. 
Then one has 
\begin{eqnarray*}
{\mathcal M}_{(M, \omega)}^{ex}(S^2)\otimes \K \cong  
\{f \in \text{\em Hom}_{\K}(W^2, \K) \mid {f_|}_{\text{\em Im} \kappa} : \text{\em Im} \kappa \to \K \ \text{is trivial} \}
\end{eqnarray*} 
as  a vector space. In particular, 
\begin{eqnarray*}
\dim{\mathcal M}_{(M, \omega)}^{ex}(S^2)\otimes \K =  \dim W^2 -2k .
\end{eqnarray*}
\end{prop}

\begin{cor}\label{cor:more_ex3} Let $({\mathbb T}^{2k}\times {\mathbb C}P(m), \omega)$ be the product of 
the symplectic manifolds with the standard symplectic forms. Then one has 
$$\dim{\mathcal M}_{({\mathbb T}^{2k}\times {\mathbb C}P(m), \omega)}^{ex}(S^2)\otimes \K = 
{}_{2k}C_2 + {}_{2k}C_4 + \cdots +{}_{2k}C_{2(\min\{m, k\})}. $$
\end{cor}

\begin{ex} 
\label{ex:Rationalized fibration}
Let us consider the fibration of the form ${\mathcal F} : ({\mathbb T}^{2k}, \omega) \to P \to S^2$, 
where $\omega$ is the standard symplectic form on ${\mathbb T}^{2k}$. 
By the Leray-Serre spectral sequence argument, 
one easily deduces that $[\omega]$ is extendable in ${\mathcal F}$ if and only if 
the rationalized fibration ${\mathcal F}_{(\K)}$ of ${\mathcal F}$ is trivial; see Appendix 
for the definition of the rationalized fibration.  

On the other hand, even up to homotopy equivalence, 
there exist infinite many distinct rationalized fibrations of appropriate fibrations over $S^2$ with fibre 
$({\mathbb T}^{2k}\times {\mathbb C}P(m), \omega)$ 
in which the symplectic class $[\omega]$ is extendable. Here  $\omega$ 
is the standard symplectic form.  To see this, we consider
the natural map 
$$
l_* : {\mathcal E}N{\mathcal W}(B) \to {\mathcal E}N{\mathcal W}(B)\otimes \K
$$
defined by $l_*(\alpha) = \alpha \otimes 1$ for $\alpha \in {\mathcal E}N{\mathcal W}(B)$, 
where $B$ is a double suspension space and $N$ is a nilpotent space. 

\begin{claim}\label{claim:eq} Let $[{\mathcal F}]$ and $[{\mathcal F}']$ be homotopy equivalence classes in ${\mathcal E}N{\mathcal W}(B)$. Then  
$l_*([{\mathcal F}]) =\l_*([{\mathcal F}'])$ if and only if $[{\mathcal F}_{(\K)}]=[{\mathcal F}'_{(\K)}]$ in 
 ${\mathcal E}N_\K{\mathcal W}(B_\K)$. 
\end{claim}

We shall prove Claim \ref{claim:eq} in Appendix. Corollary  \ref{cor:more_ex3}  implies that there is a non-trivial element 
$u$ in  the subspace  ${\mathcal M}_{(M, \omega)}^{ex}(S^2)\otimes \K$ of ${\mathcal E}M{\mathcal W}(S^2)\otimes \K$.  
Then we obtain a class $[{\mathcal F}] \in {\mathcal E}M{\mathcal W}(S^2)$ such that 
$l_*([{\mathcal F}])= mu$ for some non-zero integer $m$. Thus Claim \ref{claim:eq} yields that 
$[{\mathcal F}_{(\K)}] \neq [(n{\mathcal F})_{(\K)}]$ for any integer $n \neq 1$ 
because $l_*([{\mathcal F}])\neq \l_*(n[{\mathcal F}])$. 
\end{ex}

\begin{rem}
Suppose that $M$ is a simply-connected homogeneous space of the form $G/H$ with 
$\text{rank} \ G = \text{rank} \ H$. Then the result \cite[Theorem 1.1]{Ku2} serves to show that 
$H^*(B\aut{M})$ is a polynomial algebra. 
The proof of Theorem \ref{thm:nil-ham} below implies that $H^*(B\aut{N})$ is also a polynomial algebra 
if $N$ is a nilmanifold. We do not know a characterization for $H^*(B\aut{M})$ to be free when $M$ is a 
nilpotent ${\mathbb T}^{2k}$-separable symplectic manifold. 
\end{rem}

\begin{rem} Suppose that $B$ is $2$-connected. Then 
Remark \ref{rem:coupling} yields that for any fibration $(M, \omega) \to E \to B$, the cohomology class $[\omega]$ 
is extendable. Therefore, we have ${\mathcal M}_{(M, \omega)}^{ex}(B)={\mathcal E}M{\mathcal W}(B)$. 
\end{rem}

Theorem \ref{thm:kappa} moreover deduces important results. 
Combining the theorem with Theorem \ref{thm:M-S} we have 

\begin{cor}
\label{cor:symp}
Let ${\mathcal F}$ be a compact symplectic fibration as in Theorem
 \ref{thm:kappa}. 
Suppose that $H^*(g)\circ \kappa : H^1({\mathbb T}^{2k}) \to H^2(B)$ 
is trivial, where  $g : B \to B\aut{M}$ is the classifying map for the
 underlying fibration of  ${\mathcal F}$. 
Then there exists a compatible symplectic form on $P$; that is, the restriction of the 
form on $P$ to the fibre coincides with the given symplectic form on the fibre.   
\end{cor}

In view of the universal covering over $B\text{Symp}(M, \omega)$, we see that the classifying map from   
$B$ to $B\text{Symp}(M, \omega)$ of a symplectic bundle factors through 
$B\text{Symp}_1(M, \omega)$ if the base space 
is simply-connected. By virtue of  Theorem \ref{thm:kappa} and Theorem \ref{thm:L-M},  
we have 

\begin{cor}
\label{cor:reduction}
Let ${\mathcal F}: 
(M, \omega)  \to P \to B$ 
be a symplectic bundle over a simply-connected base and 
$f : B \to B\text{\em Symp}_1(M, \omega)$ the 
the classifying map for the principal bundle 
associated to ${\mathcal F}$. Suppose that 
$(M, \omega)$ is nilpotent and ${\mathbb T}^{2k}$-separable. 
Then the bundle ${\mathcal F}$ is Hamiltonian  
if and only if the induced map 
$H^*((Bj)\circ f)\circ \kappa : H^1({\mathbb T}^{2k}) \to H^2(B)$ is
 trivial, where $j : \text{\em Symp}_1(M, \omega) \to \aut{M}$ 
is the inclusion.    
\end{cor}

%Observe that the composite $Bj\circ f$ is the classifying map of
%${\mathcal F}$ considered an $M$-fibration. 

It is important to remark some results concerning Theorem \ref{thm:kappa} and 
Corollary \ref{cor:symp}. Geiges \cite{Ge} investigated symplectic
structures of the total spaces of 
${\mathbb T}^2$-bundles over ${\mathbb T}^2$. Compatible symplectic
structures of such bundles were considered by Kedra \cite{Ke}. 
It turns out that the total spaces of all ${\mathbb T}^2$-bundles 
over ${\mathbb T}^2$ support symplectic forms; 
see \cite[Theorems 1 and 2]{Ge} and \cite[Theorem 3.3, Remark 3.4]{Ke} for more details. 

Kahn deduced a necessary and sufficient condition for 
a symplectic torus bundle  ${\mathbb T}^2 \to E \to B$ to admit a
compatible symplectic structure provided $B$ is a surface; see
\cite[Theorem 1.1]{Ka}. Walczak proved that the total
space of torus bundle over a surface admits a symplectic structure which
is not necessarily compatible with that of the fibre if and only if 
the symplectic class of ${\mathbb T}^2$ extends to 
a cohomology class of $E$; see \cite[Theorem 4.9]{W}. 

The argument in \cite[Section 4]{Ge} implies that 
a symplectic ${\mathbb T}^2$-bundles $\xi$ over $S^2$ 
admits a compatible symplectic structure if and only if $\xi$ is
trivial; see also \cite[Proposition 1.3]{Ka} and \cite[Remark 4.13]{W}. 
We recover the fact by applying Theorem \ref{thm:kappa}; 
see Remark \ref{rem:S^2}. 
% and also Exmaple \ref{ex:principal-bd}. 

We focus on the case where the fibre is a nilmanifold. 
Let $M \to M_{\text{Ham}(M, \omega)} \to B\text{Ham}(M, \omega)$ be 
the universal Hamiltonian $M$-bundle. 
This bundle is regarded as a fibration whose classifying map is the map 
$Bj : B\text{Ham}(M, \omega) \to B\aut{M}$ induced by the inclusion 
$j : \text{Ham}(M, \omega) \to \aut{M}$; 
see \cite[Corollary 8.4]{May1}. Suppose that $(M, \omega)$ 
is the $2k$-dimensional torus with the standard symplectic form. 
Then Corollary \ref{cor:reduction} yields that the induced map 
$(Bj)^* : H^2(B\aut{M}) \to H^2(B\text{Ham}(M, \omega))$ is trivial since 
$H^2(B\aut{M})$ coincides with the image of the detective map $\kappa$; 
see Lemma \ref{lem:kappa}. This result is
generalized to the case of nilmanifolds, which are no longer 
${\mathbb T}^{2k}$-separable in general. More precisely, we establish  

\begin{thm}
\label{thm:Ham}
Let $(N, \omega)$ be a nilmanifold 
with a symplectic structure $\omega$. Then the induced map 
$(Bj)^* : H^*(B\aut{N}) \to H^*(B\text{\rm Ham}(N, \omega))$ is trivial.   
\end{thm}
  
Unfortunately, we capture no information on 
$H^*(B\text{\rm Ham}(N, \omega))$ via $H^*(B\aut{N})$ while 
a non-trivial element in $H^*(B\text{\rm Symp}(N, \omega))$ comes from 
an appropriate element in $H^*(B\aut{N})$ 
in some case; see  Example \ref{ex:torus} for example. 
Theorem \ref{thm:Ham} is deduced from Theorem \ref{thm:L-M} and the following
theorem.

\begin{thm}
\label{thm:nil-ham}
Let $(N, \omega)$ be a nilmanifold with a symplectic structure $\omega$
 and  ${\mathcal F} : 
(N, \omega)  \to P \to  B$ a fibration over a simply-connected space $B$ 
with fibre $(N, \omega)$. Let  $f : B \to B\aut{N}$ be the classifying map for 
${\mathcal F}$. 
Then the class $[\omega] \in H^2(N; {\mathbb R})$ extends to 
a cohomology class of $P$ if and only if the induced map 
$H^*(f) : H^*(B\aut{N}) \to H^*(B)$ is trivial.  
\end{thm}

An outline of this paper is as follows. 
In Section 2, 
we recall briefly a model for the evaluation map of a function space 
from \cite{B-M}, \cite{H-K-O} and \cite{Ku1} of which we make extensive use . Section 3 is devoted to
proving Theorem \ref{thm:kappa}, Propositions \ref{prop:more_ex1},  \ref{prop:more_ex2} 
and Corollary \ref{cor:more_ex3}.   
In Section 4, by considering the homomorphism induced by the evaluation
map $\aut{N} \to N$ on the fundamental group, 
we prove Theorem \ref{thm:nil-ham}.  
In Appendix, we deal with the
extension problem of a symplectic class in a fibration without assuming 
the nilpotentness of the fibre. 
%A generalization of Theorem \ref{thm:kappa} is obtained; 
%see Theorem \ref{thm:kappa'} and Remark \ref{rem:generalization}.   

\section{A model for a function space}

For the convenience of the reader and to make notation more precise, 
we recall from \cite{B-S} and \cite{Ku1} a Sullivan model for a function
space and a model for the evaluation map.    
We shall use the same terminology as in \cite{B-G} and \cite{F-H-T}.

Let $(B, d_B)$ be a connected, locally finite DGA and $B_*$ denote
the differential graded coalgebra defined by $B_q=\mbox{Hom}_\K(B^{-q}, \K)$
for $q\leq 0$ together with the coproduct $D$ and the differential $d_{B
*}$ which are dual to the multiplication of $B$ and to the differential
$d_B$, respectively. 
We denote by $I$ the ideal of the free algebra 
$\wedge(\wedge V \otimes B_*)$ 
generated by $1\otimes 1_* -1$ and all elements of the form 
$$
a_1a_2\otimes \beta - 
\sum_i(-1)^{|a_2||\beta_i'|}(a_1\otimes \beta_i')(a_2\otimes \beta_i''), 
$$
where $a_1, a_2 \in \wedge V$, $\beta \in B_*$ and $D(\beta) 
= \sum_i\beta'_i\otimes  \beta''_i$. 
Observe that $\wedge(\wedge V \otimes B_*)$ is a DGA with the differential 
$d := d_A\otimes 1\pm 1\otimes d_{B *}$. 

The result \cite[Theorem 3.5]{B-S} implies that the composite 
$$
\rho : \wedge(V \otimes B_*)  \hookrightarrow 
\wedge(\wedge V \otimes B_*) \to 
\wedge(\wedge V \otimes B_*)/I
$$ 
is an isomorphism of graded algebras. Moreover it follows from 
\cite[Theorem 3.3]{B-S} that $I$ is a differential
ideal; that is, $dI \subset I$.  
We then define a DGA of the form 
$(\wedge(V \otimes B_*), \delta = \rho^{-1}d\rho)$.  
Observe that, for an element $v \in V$ and a cycle $e \in B_*$,  
if  $d(v) = v_1\cdots v_m$ with $v_i \in V$ and 
$D^{(m-1)}(e) = \sum_je_{j_1}\otimes \cdots \otimes e_{j_m}$,
then
$$
\begin{array}{lcl}
\delta(v\otimes e) 
&=&\sum_j\pm (v_1\otimes e_{j_1}) \cdots (v_m\otimes e_{j_m}).  
\end{array}
\eqnlabel{add-1}
$$
Here the sign is determined by the Koszul  rule that   in a graded algebra 
$ab = (-1)^{\deg a \deg b} ba$.  

Let $A_{PL}$ be the simplicial commutative cochain algebra of polynomial
differential forms with coefficients in ${\mathbb Q}$; see
\cite{B-G} and  \cite[Section 10]{F-H-T}.  
Let $\mathcal{A}$ and  $\Delta {\mathcal S}$ be the category of DGA's
and that of simplicial sets, respectively.  
For $A, B\in  ob\mathcal{A}$, let $\mbox{DGA}(A, B)$ denote the set of DGA
maps from $A$ to $B$.
Following Bousfield and Gugenheim \cite{B-G}, we define functors 
$\Delta : \mathcal{A} \to  \Delta {\mathcal S}$ 
and $\Omega : \Delta {\mathcal S} \to  \mathcal{A}$
by $ \Delta (A)= \mbox{DGA}(A, A_{PL})$ and by 
$\Omega (K) = \mbox{Simpl}(K, A_{PL})$, respectively.  

Let $S_*(U)$ denote the singular
simplicial set associated with  a space $U$.   
Let $A_{PL}(U)$ be the DGA  
of polynomial differential forms on a space $U$, namely
$A_{PL}(U)=\Omega S_*(U)$. 
For spaces $X$ and $Y$, let $\F(X, Y)$ stand for the space of all continuous
maps from $X$ to $Y$. The connected component of $\F(X, Y)$ containing a
map $f : X \to Y$ is denoted by $\F(X, Y; f)$. We observe that $\aut{M}$ is 
nothing but the function space $\F(M, M; id_{M})$.  
Let 
$\alpha : A=(\wedge V, d)\stackrel{\simeq}{\to} A_{PL}(Y)$  
be a minimal model for $Y$ and 
$\beta : (B, d) \stackrel{\simeq}{\to} A_{PL}(X)$ 
a Sullivan model for $X$ for which $B$ is connected and locally
finite.

We choose a basis $\{a_k, b_k,  c_j \}_{k,j}$ 
for $B_*$ so that $d_{B_*}(a_k)=b_k$, 
$d_{B_*}(c_j)=0$ and $c_0 = 1$.
Moreover we take a basis $\{v_i\}_{i\geq 1}$ for $V$ which satisfies the condition that 
$\deg v_i \leq \deg v_{i+1}$ and $d(v_{i+1})\in \wedge V_i$ for any $i$, where
$V_i$ is the subspace spanned by the elements $v_1 ,..., v_i$.
The result \cite[Lemma 5.1]{B-S} ensures that there exist free algebra
generators $w_{ij}$, $u_{ik}$ and $v_{ik}$ of $\wedge(V\otimes B_*)$ such that  \\
(2.2) $w_{i0}= v_i\otimes 1_*$  and
$w_{ij}= v_i \otimes c_j + x_{ij}$, where $x_{ij}\in \wedge(V_{i-1} \otimes
B_*)$, \\
(2.3) $\delta w_{ij}$ is decomposable 
in $\wedge(\{w_{sl} ; s < i\})$ and \\
(2.4) $u_{ik}= v_i\otimes a_k$ and ${\delta}u_{ik}=v_{ik}$.

\noindent
Thus we have an inclusion 
$$
\gamma : E:=(\wedge(w_{ij}), \delta) \hookrightarrow  
(\wedge(V\otimes B_*), \delta)
\leqno{(2.5)}
$$ 
which is a homotopy equivalence with a retract 
$$
r :  (\wedge(V\otimes B_*), \delta) \to E. 
\leqno{(2.6)} 
$$
We refer the reader to \cite[Lemma 5.2]{B-S} for more details. 
Let $q$ be a Sullivan representative for a map $f : X \to Y$; that is,
$q$ fits into the homotopy commutative diagram 
$$
\xymatrix@C30pt@R15pt{
B \ar[r]^(0.45){\simeq} & A_{PL}(X) \\
 \wedge V \ar@{>}[u]^q \ar[r]_(0.45){\simeq} & A_{PL}(Y) . 
  \ar[u]_{A_{PL}(f)}
}
$$ 
Moreover we define a DGA map 
$\widetilde{u} :  \wedge(\wedge V\otimes B_*)/I \to \K$ 
by
$$
\widetilde{u}(a\otimes b)=(-1)^{\tau(|a|)}b(q(a)), 
\leqno{(2.7)} 
$$
where $a \in \wedge V$, $b \in B_*$ and $\tau(n)=[\frac{n+1}{2}]$. 
With the functor $\Delta : {\mathcal A} \to \Delta {\mathcal S}$
mentioned above, we put $u= \Delta(\gamma)\widetilde{u}$, where 
$\widetilde{u}$ is regarded as a $0$-simplex in 
$\Delta(\wedge(\wedge V\otimes B_*)/I)$.  
Let $M_u$ be the ideal of $E$ generated by the set
$$
\{\omega \mid \deg \omega < 0\} \cup \{\delta\omega \mid \deg \omega =
0\}\cup
\{\omega - u(\omega)\mid \deg \omega  = 0\}.
$$
We assume that, for a function space $\F(X, Y)$ which we deal with,   
$$
X  \ \text{is connected finite CW complex and} \  Y \  \text{is a nilpotent space or} 
\leqno{(2.8)}
%\eqnlabel{add-2}
$$
$$
Y \ \text{is a rational space and}  \dim \oplus_{q\geq 0}H^q(X; \K)< \infty \  \text{or}  
\dim \oplus_{i\geq 2}\pi_i(Y)< \infty. 
\leqno{(2.9)}
%\eqnlabel{add-2}
$$
Then the result \cite[Theorem 6.1]{B-S} yields that the DGA $(E/M_u, \delta)$ 
is a model for a connected component of the function space $\F(X, Y)$ containing $f$. 
Observe that, by forming the quotient $E/M_u$, one eliminates all elements of negative 
degree. Moreover an element $\omega$ of degree $0$ becomes a cycle, 
identified with the scalar $u(\omega)$.  

The proofs of \cite[Proposition 4.3]{Ku1} and \cite[Remark 3.4]{H-K-O}
allow us to deduce the following proposition; see also \cite{B-M}.

\begin{prop}\label{prop:ev-model} 
Let $\{b_j\}$ and $\{b_{j*}\}$ be 
a basis of $B$ and its dual basis of  $B_*$, respectively. 
Under the assumption (2.8) or (2.9), %with the same notation as above, 
we define a map  
$m(ev) : A=(\wedge V, d) \to  (E/M_u, \delta)\otimes B$ 
by  
$$m(ev)(x)= \sum_{j}(-1)^{\tau(|b_j|)}
\pi\circ r(x\otimes b_{j *})\otimes b_j,    
$$
for $x \in A$.  
Then $m(ev)$ is a model for the evaluation map 
$ev : \F(X, Y; f) \times X \to Y$; that is, there exists a homotopy
 commutative diagram 
$$
\xymatrix@C30pt@R15pt{
A_{PL}(Y) \ar[r]^(0.4){A_{PL}(ev)}  &  A_{PL}(\F(X, Y; f) \times X)  \\
  & A_{PL}(\F(X, Y; f)\otimes A_{PL}(X) \ar[u]_{\simeq}\\
 A \ar@{>}[uu]_{\simeq}^{\alpha} \ar[r]_(0.4){m(ev)} & 
 (E/M_u, \delta)\otimes B  \ar[u]^{\simeq}_{\xi \otimes \beta}
}
$$
in which 
$\xi : (E/M_u, \delta) \stackrel{\simeq}{\to} A_{PL}(\F(X, Y; f))$ 
is the Sullivan model for $\F(X, Y; f)$ due to Brown and Szczarba
 \cite{B-S}.   
\end{prop}

We conclude this section by recalling a result due to Gugenheim and May,
which is used in the proof of Theorem \ref{thm:kappa}.  

\begin{prop}\cite[Corollary 3.12]{G-M}
\label{prop:G-M}
Let ${\mathcal G}$ be a topological monoid with identity and $\{E_r, d_r\}$ the
 Eilenberg-Moore spectral sequence converging to $H^*(B{\mathcal G})$. 
Let $\sigma : H^*({\mathcal G}) \to E_1^{1, *}=B^{1,*}(\K, H^*({\mathcal
 G}), \K)$ be a map
 defined by $\sigma(x)=[x]$, where $B^{*,*}(\K, H^*({\mathcal G}), \K)$
 denotes the bar complex. Then the additive relation 
$$
\xymatrix@C15pt@R10pt{
H^{*+1}(B{\mathcal G}) \cong F^1\text{\em Cotor}_{C^*({\mathcal G})}(\K,
 \K) 
\ar@{->>}[r] & E^{1, *}_\infty  \ 
 \ar@{>->}[r] & \cdots  \ \ar@{>->}[r] & E_1^{1, *}  
       & H^*({\mathcal G}) \ar[l]_(0.5){\sigma}
}
$$
coincides with the cohomology suspension 
$\xymatrix@C10pt
{\sigma^* : H^{*+1}(B{\mathcal G}) \ar[r]^(0.55){\pi^*}  &
 H^{*+1}(E{\mathcal G}, {\mathcal G}) & H^*({\mathcal G}) 
\ar[l]_(0.4){\delta}^(0.4){\cong}
}$. 
\end{prop}

Originally, this result is proved for the homology spectral sequence  
in the case where ${\mathcal G}$ is a topological group. However one can
use the notion of the classifying space of a monoid due to May
\cite{May1} to construct the Eilenberg-Moore spectral sequence 
mentioned above.  
The dual argument of the proof does work well to prove 
Proposition \ref{prop:G-M}.

\section{Proofs of Theorem \ref{thm:kappa},  Propositions \ref{prop:more_ex1} and \ref{prop:more_ex2}}

In order to prove Theorem \ref{thm:kappa}, we begin with a consideration on  
a minimal model for a ${\mathbb T}^{2k}$-separable symplectic manifold. 

Let $M$ be a ${\mathbb T}^{2k}$-separable symplectic
manifold which admits a minimal model of the form described 
in Definition \ref{defn:separable}. 
Since $H^*(M; \K)$ and $H^*({\mathbb T}^{2k})$ satisfy 
the  Poincar\'e duality,  so does $H^*(\wedge Z, d)$. 
Therefore if $M$ is a $2n$-dimensional symplectic manifold, then 
$H^*(\wedge Z, d)$ is a $2(n-k)$-dimensional Poincar\'e duality
algebra. We see that 
$$
0\neq [\omega]^n = q^mq_1\cdots q_k\frac{n!}{m!}
\gamma([\beta^m(t_{1_1}t_{1_2})\cdots (t_{k_1}t_{k_2})])  
$$ 
in $H^*(M; {\mathbb R})$, where $m = n -k$. 
Thus the element $[\beta^m]$ is the top class of
$H^*(\wedge Z, d)$.  This enables us to construct a minimal model
$(\wedge V, d)$ for $M$ of the form 
$$
\bigotimes_{i=1}^k(\wedge (t_{i_1}, t_{i_2})), 0)
\otimes (\wedge (\beta, y, ..), d)
$$ 
with $d(y) = \beta^{m+1}$. It follows from \cite[Theorem 1.1]{L-S} 
that there exist a simply-connected Poincar\'e duality DGA $(C, d)$ and a 
quasi-isomorphism 
$(\wedge (\beta, y, ..), d)  \to (C, d)$. 
Thus we have a quasi-isomorphism 
$$\eta : (\wedge V, d) \to 
\bigotimes_{i=1}^k(\wedge (t_{i_1}, t_{i_2}), 0) \otimes (C, d)=:(B, d). 
$$
Since $(C, d)$ is a Poincar\'e duality algebra, 
it follows that  $C^{2m-1}= 0$.   
This implies that $d_*(\eta(\beta^m)_*)=0$ and  
$d_*(\eta(\beta^mt_{i_\lambda})_*)=0$. It is evident that 
$\eta(\beta^m)_*$ and $\eta(\beta^mt_{i_\lambda})_*$ are non-exact in
$B_*$.

By using the minimal model $(\wedge V, d)$, we construct the
model $E/M_u$ for $\aut{M}$ as in Section 2. Observe that, in this case,  the DGA map $q$ in
(2.7) is the identity map. Thus we have   
\begin{eqnarray*}
E/M_u = \wedge (y\otimes 1_*, ..., 
y\otimes \eta^{\vee}((\eta\beta^{m-1})_*), y\otimes \eta^{\vee}((\eta\beta^m)_*), 
\beta\otimes 1_*,  t_{i_1}\otimes 1_*, t_{i_2}\otimes 1_*,  ...  )
\end{eqnarray*}
with $\delta(y\otimes \eta^{\vee}((\eta\beta^m)_*))= p \beta\otimes 1_*$ 
for some non-zero rational number $p$, 
where $\eta^\vee$ denotes the dual homomorphism 
$B_* \to (\wedge V)^\vee=\text{Hom}_\K(\wedge V, \K)$ of $\eta$; see (2.2), (2.3), (2.4), (2.5) and (2.6).  
In fact, we can obtain the explicit form on the differential as follows: 
Let $\Delta$ be the coproduct on $C_*$ which is the dual to the multiplication of $C$. 
Let $\Delta^{(m)} : C_* \to C_*^{\otimes (m+1)}$ be
the iterated coproduct defined by 
$\Delta^{(m)} = (\Delta\otimes 1^{\otimes m-1})\circ \cdots\circ  (\Delta\otimes 1)\circ \Delta$. 
Then we see that 
\begin{eqnarray*}
\Delta^{(m)}(\eta \beta^m)_* 
&=& 1\otimes (\eta \beta)_*\otimes \cdots \otimes (\eta \beta)_* + 
 (\eta \beta)_*\otimes 1\otimes(\eta \beta)_* \otimes \cdots \otimes (\eta
 \beta)_* \\
&&+ (\eta \beta)_*\otimes \cdots \otimes (\eta \beta)_*\otimes 1 + \ 
\text{other terms}. 
\end{eqnarray*}
Thus formula (2.1) enables us to conclude that in $E/M_u$ 
\begin{eqnarray*}
\delta(y\otimes \eta^{\vee}(\eta\beta^m)_*) & = & 
\beta^m\otimes \eta^{\vee}(\eta\beta^m)_* \\
& = & \beta \otimes \cdots 
\otimes \beta \cdot \eta^{\vee}(\Delta^{(m)})(\eta\beta^m) \\
& = & (m+1)
\beta\otimes 1_* \cdot \beta\otimes \eta^{\vee}(\eta \beta)_* \cdots 
\beta\otimes \eta^{\vee}(\eta \beta)_* \\
&=& (-1)^{m\tau(|\beta|)}(m+1)\beta \otimes 1_*. 
\end{eqnarray*}

We observe that elements $\beta\otimes \eta^{\vee}(\eta t_{i_\lambda})_*$ of degree $1$  do not survive in $E/M_u$. 
Indeed, the same computation as above allows us to deduce that 
$$\delta(y\otimes \eta^{\vee}(\eta\beta^mt_{i_\lambda})_*)= p'_i  \beta\otimes \eta^{\vee}(\eta t_{i_\lambda})_*
$$
for some non-zero rational number $p'_i$. Thus we have 
\begin{lem} \label{lem:Mu} $\beta\otimes \eta^{\vee}(\eta t_{i_\lambda})_*  \in M_u$. 
\end{lem}

%& & \quad  y\otimes \eta^{\vee}(\eta (\beta^{m}t_{i_1}))_*, 
%y\otimes \eta^{\vee}(\eta (\beta^mt_{i_2}))_*, 
%\beta\otimes(t_{i_1})_*, \beta\otimes(t_{i_2})_* \\

\begin{ex}\label{ex:more_ex}
Let $M$ denote the product space ${\mathbb T}^{2k}\times {\mathbb C}P(m)$. 
A minimal model $(\wedge V, d)$ for $M$ is given by 
$$
\bigotimes_{i=1}^k(\wedge (t_{i_1}, t_{i_2}), 0)\otimes (\wedge (\beta, y), dy=x^{m+1}). 
$$ 
We define a quasi-isomorphism 
$\eta : (\wedge V, d) \to \bigotimes_{i=1}^k(\wedge (t_{i_1}, t_{i_2}), 0)\otimes \K[\beta]/(\beta^{m+1})$ 
by $\eta(t_{i_j})=t_{i_j}$, $\eta(\beta)=\beta$ and $\eta(y)=0$. The construction of the model for a function space   
described in Section 2 enables us to obtain an explicit model $E/M_u$ for $\aut{M}$. In particular, 
we see that  
\begin{eqnarray*}
(E/M_u)^1 \cong \K\{ t_{i_1}\otimes 1_*, t_{i_2}\otimes 1_*,  y\otimes \eta^{\vee}((\eta \beta^m)_*), \\
&& \hspace{-5cm}  y\otimes \eta^{\vee}((\eta \beta^{m-1}t_{j_1}t_{j_2})_*), ..., 
y\otimes \eta^{\vee}((\eta \beta^{m-q}t_{j_1}\cdots t_{j_{2q}})_*) ;  \\
&& \hspace{-0.5cm} j_u\neq j_v \ \text{if} \ u\neq v, 1\leq i_1, i_2 \leq k  
\}
\end{eqnarray*}
where $q = \min\{m, k\}$.  Moreover it follows that $\delta(y\otimes \eta^{\vee}((\eta \beta^m)_*))=p\beta\otimes 1_*$ for some 
$p\neq 0$ and $\delta (t_{i_\lambda}\otimes 1_*)=0$ for $\lambda = 1, 2$. Lemma \ref{lem:Mu} implies that  
$\delta(y\otimes \eta^{\vee}((\eta \beta^{m-s}t_{j_1}\cdots t_{j_{2s}})_*))=0$ in $E/M_u$ for $1\leq s\leq q$. 
\end{ex}

Let $\{E_r, d_r\}$ and $\{_\K E_r, d_r\}$ be 
the Leray-Serre spectral sequences of the fibration  
$(M, \omega) \to P \to B$ with coefficients in the real number field and
that with coefficients in $\K$, respectively.   
By relying on  Propositions \ref{prop:A} and 
\ref{prop:B} below, we shall prove Theorem \ref{thm:kappa}.

\begin{prop}
\label{prop:A}
The element $\beta \in \ \!\!  _{\K}E_2^{0,2}$ is a permanent cycle.
\end{prop}

\begin{prop}
\label{prop:B}  The composite $H^*(f)\circ \kappa$ is trivial if and only if
$d_2(t_{i_1}t_{i_2})=0$ in  $_{\K}E_2^{*,*}$ for any $i$. 
\end{prop}

\medskip
\noindent
{\it Proof of Theorem \ref{thm:kappa}}.
The inclusion $\K \to {\mathbb R}$ induces the morphism of spectral
sequences $\{\gamma_r\} : \{ _\K E_r, d_r\} \to \{E_r, d_r\}$ for which 
$\gamma_2$ is injective.
   
%Proposition \ref{prop:A} implies that $\gamma_2(\beta_j)$ is a permanent
%cycle for any $j$ and hence so is $\beta \in E_2^{0,2}$. 
By virtue of Proposition \ref{prop:A}, we have 
\begin{eqnarray*}
d_2(\omega)  &=&  \gamma_2 
d_2(q\beta + \sum_{i=1}^kq_i t_{i_1}t_{i_2})
=\sum_{i=1}^kq_i\gamma_2d_2(t_{i_1}t_{i_2}) \\ 
&= & \sum_{i=1}^kq_i\gamma_2 \Bigl(d_2(t_{i_1})t_{i_2} 
 - t_{i_1}d_2(t_{i_2})\Bigr). 
\end{eqnarray*}
Thus it follows from Proposition \ref{prop:B} that 
$d_2(\omega)= 0$ if and only if $H^*(f)\circ \kappa$ is trivial. 
Observe that $q_i\neq 0$ for any $i$. 
Suppose that $d_2(\omega)=0$. Then the argument in Remark \ref{rem:coupling}
yields that $d_3(\omega)= 0$. This completes the proof. 
\hfill\qed

\medskip
In order to prove Propositions \ref{prop:A} and \ref{prop:B}, 
we use two spectral sequences. 
Let $\{_{EM}E_r, d_r\}$ be the Eilenberg-Moore spectral sequence 
with coefficients in $\K$ converging to $H^*(M_{\aut{M}})$ 
with 
$$
_{EM}E_2^{*,*}\cong \text{Cotor}_{H^*(\text{aut}_1(M))}^{*,*}(\K, H^*(M))
$$
as an algebra. 
Let $\{_\K \widetilde{E}_r, \widetilde{d}_r\}$ be the Leray-Serre
spectral sequence of the universal $M$-fibration  
$M \stackrel{i}{\to} M_{\aut{M}} \stackrel{\pi}{\to} B\aut{M}$ 
with coefficients in $\K$. Observe that 
$$
_\K \widetilde{E}_2^{i,j}\cong  H^i({B\aut{M}})\otimes H^j(M)
$$ as a bigraded algebra 
since $B\aut{M}$ is simply-connected.

\medskip
\noindent
{\it Proof of Proposition \ref{prop:A}}.
Let us consider the Eilenberg-Moore spectral sequence 
$\{_{EM}E_r, d_r\}$ mentioned above. Then the differential 
$$
d_1 : \ _{EM}E_1^{0,*}=H^*(M) \to \ _{EM}E_1^{1,*}
=\overline{H^*(\aut{M})}\otimes H^*(M)
$$
is induced by the evaluation map; that is, $d_1(x)= - ev^*(x)$ for any 
$x \in H^*(M)$, where $\overline{H^*(\aut{M})}=H^*(\aut{M})/\K$; see 
\cite{May1}.  
We use here the normalized cobar construction to compute
the cotorsion product. 
Then Proposition \ref{prop:ev-model} and the explicit model $(E/M_u, \delta)$ 
for $\aut{M}$ mentioned above enable us to deduce that 
$$
d_1(\beta) = - ev^*(\beta) =- \Bigl([\beta\otimes 1_*]1 +
\sum_{i=1}^k\sum_{\lambda=1}^2[\beta\otimes
(t_{i_\lambda})_*]t_{i_\lambda}\Bigr) = 0.  
$$
For dimensional reasons, it is readily seen that the element 
$\beta \in \ \! _{EM}E_2^{0,2}$ is a permanent cycle. 
The induced map $i^*$ fits into the
following commutative diagram; see \cite[Corollary 3.11]{G-M}. 
$$
\newdir{ >}{{}*!/-10pt/@{>}}
\xymatrix@C20pt@R15pt{
H^*(M_{\aut{M}}) \ar@{->>}[d] \ar[rr]^{i^*}   & & H^*(M) \\
_{EM}E_\infty^{0,*} \ar@{ >->}[r] & \cdots  \ar@{ >->}[r] & 
_{EM}E_2^{0.*}=\text{Cotor}_{H^*(\aut{M})}^{0,*}(\K, H^*(M)).  \ar@{>->}[u]
}
\leqno{(3.1)} 
$$
We see that the map $i^*$ factors through  the edge homomorphism of the
spectral sequence $\{_{EM}E_r, d_r\}$. 
This implies that there is an element $\beta' \in H^2(M_{\aut{M}})$ such
that $i^*(\beta') = \beta$. 
Moreover the map $i^*$ also coincides with the edge
homomorphism 
$$
\xymatrix@C20pt@R10pt{
H^*(M_{\aut{M}}) \ar@{->>}[r] & _\K \widetilde{E}_\infty^{0,*} \ 
\ar@{>->}[r] & \cdots \  \ar@{>->}[r] & 
_\K \widetilde{E}_2^{0.*} \cong H^*(M)
}
$$
of the Leray-Serre spectral sequence. 
Thus $\beta \in \! \ _\K\widetilde{E}_2^{0,2}$ is a permanent cycle and hence 
so is $\beta \in \ \! _\K E_2^{0,2}$ for the naturality of the Leray-Serre
spectral sequence. We have the result.  
\hfill\qed
\medskip

Let $\{_{EM}\widehat{E}_r, \widehat{d}_r\}$ be 
the Eilenberg-Moore spectral sequence with 
$$
_{EM}\widehat{E}_2^{*,*}\cong \text{Cotor}_{H^*(\aut{M})}^{*,*}(\K, \K)
$$
converging to $H^*(B\aut{M})$
which is described in Proposition \ref{prop:G-M}. 
Since the elements $t_{i_1}\otimes 1_*$ and $t_{i_2}\otimes 1_*$ are
primitive in $H^1(\aut{M})$, it follows that 
$[t_{i_1}\otimes 1_*]$ and $[t_{i_2}\otimes 1_*]$ survive at the
$E_2$-term. 
Moreover, by dimensional reasons, the elements 
$[t_{i_1}\otimes 1_*]$ and $[t_{i_2}\otimes 1_*]$ in 
$_{EM}\widehat{E}_2^{*,*}$ are non-exact permanent cycles and 
linearly independent.
Thus we see that 
$\K\{[t_{i_1}\otimes 1_*], [t_{i_2}\otimes 1_*]
 ; 1\leq i \leq k\}$ is a subspace of 
$H^2(B\aut{M})$. 

We provide a lemma to prove Proposition \ref{prop:B}.

\begin{lem}
\label{lem:kappa}
$\text{\em Im} \ \kappa = \K\{[t_{i_1}\otimes 1_*], [t_{i_2}\otimes 1_*]
 ; 1\leq i \leq k\}.$  
\end{lem}

\begin{proof}
By virtue of Proposition \ref{prop:ev-model}, we see that 
$$ev^*\circ (p_M)^*(t_i) = t_i\otimes 1_* \in
 H^*(\aut{M})=H^*(E/M_u,  \delta). 
$$
By definition, the transgression  
$\tau :  H^1(\aut{M}) \to H^2(B\aut{M})$ 
is the additive relation 
$\xymatrix@C18pt{
H^1(\aut{M}) \ar[r]^(0.35){\delta}_(0.35){\cong}  &
 H^2(E\aut{M}, \aut{M}) & H^2(B\aut{M}) 
\ar[l]_(0.4){\pi^*}
}
$. 
 The result follows from Proposition \ref{prop:G-M}. 
\end{proof}

%\begin{lem}
%The element $\omega \in E_2^{0, 2}$ is a permanent cycle if and only if
% the image of 
%the vector space $\K\{[t_1\otimes 1_*], [t_2\otimes 1_*]\}$ by
%$H^*(f)$ is trivial.  
%\end{lem}

\medskip
\noindent
{\it Proof of Proposition \ref{prop:B}}.
Recall the Eilenberg-Moore spectral sequence $\{_{EM}E_r, d_r\}$
mentioned above. Then we see that $d_1(t_{i_\lambda})= - ev^*(t_{i_\lambda})
=-  [t_{i_\lambda}\otimes 1_*]1 \neq 0$ in 
$_{EM}E_1^{1,*}$. Thus the diagram (3.1) implies that $(\text{Im} \ i^*)^1=0$. 
Moreover the naturality of the Eilenberg-Moore spectral sequence allows
us to conclude that 
$$\pi^*(\K\{[t_{i_1}\otimes 1_*], [t_{i_2}\otimes 1_*] ; 1\leq i\leq
k\})=0, 
$$
where $\pi : M_{\aut{M}} \to B\aut{M}$ is 
the projection of the universal $M$-fibration.  
Then it follows that in the Leray-Serre spectral sequence $\{_\K \widetilde{E}_r, \widetilde{d}_r\}$, 
the elements $\widetilde{d}_2(t_{i_1})$ and $\widetilde{d}_2(t_{i_2})$ for 
$1\leq i\leq k$ are linearly independent and that 
$$\K\{\widetilde{d}_2(t_{i_1}), \widetilde{d}_2(t_{i_2}) ; 1\leq i\leq k\} 
= \K\{[t_{i_1}\otimes 1_*], [t_{i_2}\otimes 1_*]; 1\leq i\leq k\}.
$$ 
Therefore it is immediate that $d_2(t_{i_1}t_{i_2})=0$ in $_\K E_2^{0,2}$ if 
$H^*(f)\circ \kappa= 0$. 

Suppose that $d_2(t_{i_1}t_{i_2})=0$ for any $i$. 
Then we have 
$$0=d_2(t_{i_1}t_{i_2})
= H^*(f)(\widetilde{d}_2(t_{i_1}))\otimes t_{i_2} 
 - H^*(f)(\widetilde{d}_2(t_{i_2}))\otimes t_{i_1}.
$$    
It follows from Lemma \ref{lem:kappa} that $H^*(f)\circ \kappa$ is trivial. 
This completes the proof. 
\hfill\qed

\medskip
Theorem \ref{thm:kappa} is the key to proving Propositions \ref{prop:more_ex1} and \ref{prop:more_ex2}. 

\medskip
\noindent
{\it Proof of Proposition \ref{prop:more_ex1} .} We recall the product on 
${\mathcal E}M{\mathcal W}(B)$. For any classes $[{\mathcal F}]$ and $[{\mathcal F}']$ in ${\mathcal E}M{\mathcal W}(B)$, 
we see that $[{\mathcal F}]\ast [{\mathcal F}'] =[{\mathcal F}_{f\ast f'}]$, where $f$ and $f' : B \to B\aut{M}$ are classifying maps 
for $[{\mathcal F}]$ and $[{\mathcal F}']$, respectively. 
The induced map 
$$
H_*(f\ast f') : \widetilde{H}_*(B) \stackrel{\Delta_*}{\longrightarrow}  \widetilde{H}_*(B)\oplus \widetilde{H}_*(B) 
\stackrel{H_*(\nabla\circ f\vee f')}{\longrightarrow}  \widetilde{H}_*(B\aut{M})
$$
coincides with $H_*(f) + H_*(f')$ since $\Delta_*(a) = (a, a)$ for $a \in \widetilde{H}_*(B)$. 
Thus it follows that $H^*(f\ast f') = H^*(f) + H^*(f')$ in the cohomology.  Suppose that $[{\mathcal F}]$ and $[{\mathcal F}'] $ are in 
${\mathcal M}_{(M, \omega)}^{ex}(B)$. Then Theorem \ref{thm:kappa} yields that $H^*(f \ast f')\circ \kappa= 
(H^*(f) + H^*(f'))\circ \kappa = 0$. 
By using Thereom \ref{thm:kappa} again, we conclude that 
$[{\mathcal F}_{f\ast f'}]$ is in ${\mathcal M}_{(M, \omega)}^{ex}(B)$.
\hfill\qed

%\medskip
%Let $(\wedge V, d)$ be a minimal model for a simply-connected space $X$. 
%Before proving Proposition \ref{prop:more_ex2}, 
%we recall the isomorphism 
%$\nu : V \to \text{Hom}_{\mathbb Z}(\pi_*(X), \K)$ defined by 
%$$\nu(v)([\alpha])e^n=Q(\alpha)v,
%$$ 
%where $e^n$ is the base of $H^n(S^n)$ and $Q(\alpha)$ denotes the indecomposable part 
%of a model for a map $\alpha : S^n \to X$; see \cite[Theorem 2]{B-G} and \cite[Theorem 15.11]{F-H-T}.  
%Observe that $\nu$ is natural with respect to continuous maps. 

\medskip
\noindent
{\it Proof of Proposition \ref{prop:more_ex2}.} We first recall that the suspension space $B$ is formal; see 
\cite[Theorem 13.9]{F-H-T}.   
Let $(\wedge \widetilde{W}, d)$ be a minimal model for $B\aut{M}$.   
By virtue of the Sullivan-de Rham equivalence theorem \cite{B-G}, we have a well-defined bijection 
$$
\Theta : [B, (B\aut{M})_\K] \to \text{Hom}_{DGA}( (\wedge \widetilde{W}, d), (H^*(B), 0))/\simeq , 
$$
which sends a class $[f]$ to the homotopy class of a Sullivan representative for $f$.  
%Here  $\simeq$ denotes the homotopy relation between morphisms of DGA's. 

We shall show the former half of the assertions. Assume that $H^*(B\aut{M})$ is a free algebra. 
Then it is readily seen that  $(\wedge \widetilde{W}, d)= (\wedge W, 0)$.  
The result \cite[Lemma 14.5]{F-H-T} yields that 
if $\varphi_0$ and $\varphi_1$ are  morphisms of DGA's from a Sullivan algebra with $\varphi_0 \simeq \varphi_1$, then 
$\varphi_0$ is chain homotopic to $\varphi_1$.  
Thus it follows that the homotopy relation $\simeq$, which we consider here,  
is nothing but the equal relation.  Hence we see that the restriction gives rise to an isomorphism 
$$
\Theta_2: \text{Hom}_{DGA}( (\wedge W, 0), (H^*(B), 0))/\simeq \ \stackrel{\cong}{\longrightarrow} 
\text{Hom}_{GV}( W, H^*(B)). 
$$
Observe that the bijection 
$\Theta_2\circ \Theta$ sends a class $[f]$ to the restriction of the induced map $H(f) : H^*(B\aut{M})=\wedge W \to H^*(B)$ to 
the subspace $W$.

Recall the natural bijection $\theta : [B, (B\aut{M})_\K]_* \to [B, (B\aut{M})_\K]$. In view of the additive structures on 
$[B, (B\aut{M})_\K]_*$ and  $\text{Hom}_{GV}( W, H^*(B))$, we can deduce that the composite  
$
\Theta_2\circ \Theta \circ \theta :  [B, (B\aut{M})_\K]_*  \to \text{Hom}_{GV}( W, H^*(B))
$
is a morphism of abelian groups. Thus $[B, (B\aut{M})_\K]_*$ admits  a structure of vector space which respects 
the additive structure of the abelian group. In consequence,  we see that  $\Theta_2\circ \Theta \circ \theta$ is an 
isomorphism of vector spaces. 
By the same way,  $[B, (B\aut{M})_\K]$ is equipped with a structure of vector space so that  $\theta$ is an isomorphism.  

Let $l : B\aut{M} \to (B\aut{M})_\K$ be the rationalization map.
To obtain the first isomorphism, it suffices to show that 
the induced map 
$$
l_*  :  [B, (B\aut{M})] \to  [B, (B\aut{M})_\K]
$$  gives rise to 
isomorphism from $[B, (B\aut{M})]\otimes \K$ to  $[B, (B\aut{M})_\K]$ of vector spaces 
because $[B, B\aut{M}] \cong {\mathcal E}M{\mathcal W}(B)$ as an abelian group.

By assumption we see that $B=S^2\wedge B'$ for some finite CW complex  $B'$.  
Thus we have a commutative diagram 
$$
\xymatrix@C25pt@R15pt{
\pi_2(map_*(B', B\aut{M})) \ar[d]_{\pi_2(\widetilde{l})} &  [B, B\aut{M}]_*  \ar[l]_(0.4){ad}^(0.4){\cong} 
\ar[d]^{l_*}  \ar[r]^{\theta}_{\cong} & [B, B\aut{M}]  \ar[d]_{l_*} \\
\pi_2(map_*(B', (B\aut{M})_\K)) &  [B, (B\aut{M})_\K]_* \ar[l]_(0.4){ad}^(0.4){\cong} 
  \ar[r]^{\theta}_{\cong} & [B, (B\aut{M})_\K]  }
$$
in which the lower sequence is an isomorphism of vector spaces and the upper sequence is an isomorphism of abelian groups. 
Here $ad$ denotes the adjoint map and 
$
\widetilde{l} : map_*(B', B\aut{M}) \to map_*(B', (B\aut{M})_\K)
$ 
is the map between the spaces of based maps induced by the rationalization map $l$. 

The result \cite[II Theorem 3.11]{H-M-R} yields 
that the restriction of the map $\widetilde{l}$ to each connected component is the rationalization. 
Thus we see that $\pi_2(\widetilde{l})\otimes \K$ 
is an isomorphism and hence ${\mathcal E}M{\mathcal W}(B)\otimes \K \cong 
\text{Hom}_{GV}( W, H^*(B))$ as a 
vector space. 

Suppose that $B=S^2$. Since $B\aut{M}$ is simply-connected, it follows that $\widetilde{W}^1=0$ and $\widetilde{W}^2=W^2$.  
Then $dv=0$ for any $v\in \widetilde{W}^2$. Therefore we see that the restriction gives rise to an isomorphism
$$
\Theta_2: \text{Hom}_{DGA}( (\wedge \widetilde{W}, d), (H^*(B), 0))/\simeq \ \stackrel{\cong}{\longrightarrow} 
\text{Hom}_{GV}( \widetilde{W}, H^*(B))=\text{Hom}_\K(W^2, \K). 
$$
By the same argument as above, we have ${\mathcal E}M{\mathcal W}(S^2) \otimes \K\cong 
\text{Hom}_{\K}( W^2, \K)$.

Suppose  that $(M, \omega)$ is a nilpotent ${\mathbb T}^{2k}$-separable symplectic manifold and $B=S^2$. 
The second isomorphism can be extracted from the first one and Theorem \ref{thm:kappa}.  
It follows from Lemma  \ref{lem:kappa} that 
$\dim \text{Im} \kappa = 2k$. This allows us to obtain the equality on the dimension.   
\hfill\qed

\medskip
\noindent
{\it Proof of Corollary \ref{cor:more_ex3}.} Let $M$ denote the symplectic manifold 
${\mathbb T}^{2k}\times {\mathbb C}P(m)$.  
We recall the basis in $(E/M_u)^1$ described in 
Example \ref{ex:more_ex}. Then we see that 
\begin{eqnarray*}
H^1(\aut{M})= \K\{t_{i_1}\otimes 1_*, t_{i_2}\otimes 1_*,  
y\otimes \eta^{\vee}((\eta \beta^{m-s}t_{j_1}\cdots t_{j_{2s}})_*) ;  \\
&& \hspace{-7cm} j_u\neq j_v \ \text{if} \ u\neq v, 1\leq s\leq \min\{m, k\}, 1\leq i_1, i_2 \leq k  
\}.
\end{eqnarray*} 
Moreover we have $H^1(\aut{M})\cong H^2(B\aut{M})$. Then 
the result follows from an appeal to the equality in Proposition \ref{prop:more_ex2}. 
\hfill\qed

\medskip

\begin{ex} 
\label{ex:torus} 
We present here an application of  
Corollary \ref{cor:reduction}.

Let $(M\times {\mathbb T}^2, \omega + dt_1\wedge dt_2)$ be the symplectic 
 manifold for which  $dt_1\wedge dt_2$ is the standard symplectic form on
 ${\mathbb T}^2$. 
We define a monoid map $\ell : \aut{S^1} \to \aut{M\times {\mathbb T}^2}$ by 
$\ell(\gamma) = id_M\times id_{\mathbb T} \times \gamma$. 
Let $h: S^1 \to \aut{M\times {\mathbb T}^2}$ be the composite of 
$\ell$ and the monoid map $S^1 \to \aut{S^1}$ arising from 
the multiplication of the circle.  
Then $h$ factors through 
$\text{Symp}_1(M\times {\mathbb T}^2, \omega + dt_1\wedge dt_2))$.  
Let us consider the following commutative diagram 
$$\xymatrix@C20pt@R15pt{
H^1({\mathbb T}^2) \ar[r]^(0.45){p_{M\times {\mathbb T}^2}^*} \ar[rrd]_{\kappa}
 & H^1(M\times {\mathbb T}^2) 
\ar[r]^(0.45){ev^*} & H^1(\aut{M\times {\mathbb T}^2}) \ar[d]_{\tau} 
\ar[r]^(0.6){h^*} & H^1(S^1) \ar[d]^{\tau}  \\ 
 & & H^2(B\aut{M\times {\mathbb T}^2}) \ar[r]_(0.6){(B\eta)^*} &
 H^2(BS^1),  
}
$$
where $p_{M\times {\mathbb T}^2} : M\times {\mathbb T}^2 \to  {\mathbb T}^2$ 
is the projection onto the second factor.
Since $(p_{M\times {\mathbb T}^2}\circ ev\circ h)(t)= (0, t)$, 
it follows that 
$$
(p_M\circ ev \circ h)_* : \pi_1(S^1)_\K \to 
\pi_1(\aut{M\times {\mathbb T}^2})_\K \to \pi_1({\mathbb T}^2)_\K
$$ 
is injective. This yields that 
$\text{Im} \ ((Bh)^*\circ \kappa)= H^*(BS^1)$.  

For a simply-connected space $B$, 
we choose a class $a \in H^2(B; {\mathbb Z})\cong [B, BS^1]$. 
Let us consider the symplectic bundle  
$(M\times {\mathbb T}^2, \omega + dt_1\wedge dt_2) \to P \to B$ 
with the classifying map $(Bh)\circ a : B \to BS^1 \to 
B\text{Symp}_1(M\times {\mathbb T}^2, \omega + dt_1\wedge dt_2))$.  
In view of Corollary \ref{cor:reduction},  we see that 
the bundle is Hamiltonian if and only if $a=0$.  
\end{ex}

\begin{rem}
\label{rem:S^2}
Corollary \ref{cor:symp} is refined in the
 case of a fibration over $S^2$ with fiber ${\mathbb T}^2$. 
In fact, it follows that a symplectic ${\mathbb T}^2$-bundle $\xi$ 
over $S^2$ possesses a compatible symplectic structure if and only if 
 $\xi$ is trivial (compare Example \ref{ex:Rationalized fibration}). 
To see this, we recall the result due to Earle and Eells which asserts
 that the natural inclusion 
$i : {\mathbb T}^2 \to \text{Diff}_1({\mathbb T}^2)$ 
is a homotopy equivalence; see also \cite{Gr}. 
Let $g : S^2 \to B\text{Diff}_1({\mathbb T}^2)$ be the classifying map
 of the given ${\mathbb T}^2$-bundle $\xi$.  Then homotopy class $[g]$
 is identified with the element 
$(\psi^*[t_1], \psi^*[t_2])$
via the bijection 
$$
[S^2, B\text{Diff}_1({\mathbb T}^2)] \ 
\mapleftud{(Bi)_*}{\approx} \ 
[S^2, B{\mathbb T}^2]\cong H^2(S^2; {\mathbb Z}^{\oplus 2}),
$$ 
where $[t_1]$ and $[t_1]$ are generators of 
$H^2(B{\mathbb T}^2; {\mathbb Z})$ and $\psi : S^2 \to B{\mathbb T}^2$ 
denotes the map determined uniquely by $g$ 
up to homotopy under the bijection $(Bi)_*$. 

We prove the ``only if'' part. By virtue of Theorem \ref{thm:kappa}, we
 see that $g^*\circ \kappa$ is trivial. Let us consider the following
 commutative diagram.   
$$\xymatrix@C20pt@R15pt{
H^2(S^2) & H^2(B\aut{{\mathbb T}^2}) \ar[d] \ar[l]_{g^*} 
   &  H^1(\aut{{\mathbb T}^2}) 
    \ar[l]_{\tau} \ar@<1ex>[dd]^{i^*} \\
 & H^2(B\text{Diff}_1{{\mathbb T}^2}) \ar[d]_{\cong}^{(Bi)^*} 
   \ar[ul]^{g^*} &  \\
 & H^2(B{\mathbb T}^2) \ar@/^/[uul]^{\psi^*} &  
        H^1({\mathbb T}^2) \ar[l]_{\tau} 
\ar@<1ex>[uu]^{ev^*}
}
$$
Observe that $i^*\circ ev^*= id$.  
Since $H^2(S^2; {\mathbb Z})$ is torsion free, it follows  
from the definition of the detective map $\kappa$ that 
$(\psi^*[t_1], \psi^*[t_2])=(0, 0)$.   
We have the result.
The converse is immediate.  
\end{rem}

%\begin{ex}
%Let $\xi$ be a principal ${\mathbb T}^{2}$ bundle ${\mathbb T}^{2} \to E \to M$ over a closed oriented manifold $M$.  
%Then the class of the standard symplectic form on  ${\mathbb T}^{2}$ is extendable to a class of 
%$H^2(E; {\mathbb R})$ if and only if $\xi$ is trivial; see \cite[Remark 4.13]{W}.
%\end{ex}

\begin{ex} (cf. \cite[Remark 4.13]{W}) 
\label{ex:principal-bd}
Let $\xi$ be a principal ${\mathbb T}^{2}$ bundle over a
 simply-connected symplectic manifold whose second integral 
cohomology is torsion free. The same argument as in Remark \ref{rem:S^2}
 enables us to deduce that $\xi$ admits a compatible symplectic
 structure if and only if the bundle is trivial. 
\end{ex}

\section{A fibration whose fiber is a nilmanifold} 
%\section{Non ${\mathbb T}^{2k}$-separable cases}

\noindent
{\it Proof of Theorem \ref{thm:nil-ham}.}
The ``if'' part is immediate; see Remark \ref{rem:coupling}. 
In order to prove the ''only if'' part, we
 choose a minimal model $(\wedge Z, d)$ for $N$ so that 
$\wedge Z = \wedge (x_1, ...., x_n)$ and 
$d(x_i) \in \wedge (x_1, ..., x_{i-1})$, where $\deg x_i = 1$; 
see \cite{Hase}. 
By using $(\wedge Z, d)$ as a model for the source space  $N$ and taking
the same model for the target space, we construct a Sullivan model 
$E/M_u$ for ${\mathcal G}:=\aut{N}$; see Section 2. 
Observe that there exists a surjective map 
$\wedge (x_1\otimes 1_*, ...., x_n\otimes 1_*) \to E/M_u$. 
We use the same notation as $x_1\otimes 1_*$ for 
its image of the surjection.  
Choose a basis 
$\{x_{i_1}\otimes 1_*, ..., x_{i_k}\otimes 1_* \}$ 
of $(E/M_u)^1$. We then have 
$$
(E/M_u, \delta) \cong 
(\wedge(x_{i_1}\otimes 1_*, ..., x_{i_k}\otimes 1_*),
 \delta).
$$  
Since the degree of the element $x_{i_j}\otimes 1_*$ is odd for any $i_j$, 
it follows that $(E/M_u, \delta)$ is minimal. 
This implies that $\delta \equiv 0$ because $\aut{M}$ is a Hopf space. 

We describe here generators of $H^*(B{\mathcal G})$.
Let us consider the Eilenberg-Moore spectral sequence $\{E_r, d_r\}$ 
converging to $H^*(B{\mathcal G})$. 
It is immediate that each $x_{i_j}\otimes 1_*$ is primitive. 
Thus we see that 
$$
E_2^{*,*}\cong \text{Cotor}_{H^*({\mathcal G})}(\K, \K)\cong \K\bigl[ 
[x_{i_1}\otimes 1_*], ..., [x_{i_k}\otimes 1_*] \bigr],  
$$ 
where $\deg [x_{i_l}\otimes 1_*]=2$ for any $1\leq l\leq k$. 
For dimensional reasons, we have 
$$
H^*(B{\mathcal G})\cong \K\bigl[ 
[x_{i_1}\otimes 1_*], ..., [x_{i_k}\otimes 1_*] \bigr]. 
$$
%
%Suppose that  
% $(f)^*[x_{i_s}\otimes 1_*]\neq 0$ for some $s$.  

Let $X$ be a nilpotent space with a minimal model $(\wedge V, d)$.   
Let $\pi^n(\wedge V)$ denote the cohomology $H^*(Q(\wedge V), d_0)$, where 
$d_0$ denotes the linear part of the differential $d$.  
By virtue of \cite[12.7 Theorem]{B-G}, we have a natural bijection 
$$
\widetilde{\nu} : \pi_n(X)_\K \to 
\text{Hom}(\pi^n(\wedge V), \K)=\pi^n(\wedge V)^\vee  
$$
for $n \geq 1$, which is an isomorphism for $n \geq 2$. 
Here $\pi_*(X)_\K$ denotes the rationalization of $\pi_*(X)$ while $\pi_n(X)_\K=\pi_n(X)\otimes \K$ for 
$n\geq 2$.  
The proof of \cite[Theorem 2]{G} enables us to obtain the following
 commutative diagram 
$$
\xymatrix@C30pt@R15pt{
\pi_2(B{\mathcal G}) \ar[r]^{\partial'}_{\cong} & \pi_1({\mathcal G}) 
  \ar[d]^{ev_*}\\
 \pi_2(B) \ar[u]^{f_*} \ar[r]_{\partial} & \pi_1(N),  
}
$$
where  $\partial'$ is the connecting homomorphism of 
the universal fibration ${\mathcal G} \to E{\mathcal G} \to B{\mathcal
 G}$. 
Thus we obtain a commutative diagram 
$$
\xymatrix@C20pt@R15pt{
\pi_2(B{\mathcal G})_\K \ar[r]^{\partial'_\K}_{\cong}  &  
       \pi_1({\mathcal G})_\K   \ar[d]^{ev_*}
           \ar[r]_(0.4){\cong}^(0.4){\widetilde{\nu}} & 
 (\pi^1(E/M_u))^\vee  \ar[d]^{m(ev)^\vee} \\ 
 \pi_2(B)_\K \ar[u]^{f_*} \ar[r]_{\partial_\K}  
& \pi_1(N)_\K \ar[r]_(0.4){\cong}^(0.4){\widetilde{\nu}} & 
(\pi^1(\wedge Z))^\vee.    
}
$$

Suppose that $[\omega]\in H^2(M; {\mathbb R})$ is extendable and 
$f^* : H^*(B{\mathcal G})\to H^*(B)$ is non-trivial.  
By virtue of Proposition \ref{prop:ev-model}, we see that 
$m(ev)(x_{i_s})=x_{i_s}\otimes 1_*$. It follows that 
$m(ev)^\vee : \pi^1(E/M_u)^\vee \to \pi^1(\wedge Z)^\vee$ 
is injective. As considered above, 
the generators of $H^*(B{\mathcal G})$ are concentrated in 
$H^2(B{\mathcal G})$. Thus we see that the map 
$H_*(f) : H_2(B) \to H_2(B{\mathcal G})$ is non-trivial and hence 
so is $f_* : \pi_2(B)_\K \to \pi_2(B{\mathcal G})_\K$. 
This yields that the composite of the lower sequence of maps 
in the diagram above is also non-trivial.

On the other hand, since $[\omega]$ extends to a class of the cohomology of
$P$, it follows from the proof of \cite[Theorem 1.3]{St} that 
the rationalization of the fibration $N\to P \to B$ is a trivial one. 
Observe that $P$ is a nilpotent space; see Example
\ref{ex:T-separable}. 
This allows us to conclude that the connecting homomorphism 
$\partial_\K : \pi_2(B)_\K \to \pi_1(N)_\K$ is trivial, 
which is a contradiction. We have the result.        
\hfill\qed

\begin{rem}
\label{rem:Ham}
We mention that $\pi_1(\text{Ham}({\mathbb T}^2))=0$; see
 \cite[7.2]{P}. It follows from the proof of Theorem \ref{thm:nil-ham}
 that $(QH^*(B\aut{N}))^i=0$ for $i > 2$ if $N$ is a nilmanifold. Therefore
 Theorem \ref{thm:Ham} follows immediately from these results 
 in the case of the $2$-dimensional torus. 
\end{rem}

\medskip
\noindent
{\it Acknowledgments.}
I am particularly grateful to Kojin Abe 
for many valuable comments on this work. 
I thank Andr\'es Pedroza for conversation
on the group of Hamiltonian diffeomorphisms.  
I also thank the referee for comments to revise a previous version of this paper.

\section{Appendix} 

So for we assume that the fibration we deal with has the nilpotent
symplectic fibre. 
Nevertheless, some important symplectic manifold is not nilpotent; 
see Example \ref{ex:non-nil} below. 
We rephrase Theorem \ref{thm:kappa} in such a case.

Let ${\mathcal F} : M  \stackrel{i}{\to} P \stackrel{p}{\to} B$ be a
fibration over a simply-connected space and 
$m : (\wedge V, d) \stackrel{\simeq}{\to} A_{PL}(B)$ be a minimal model
for $B$. Then we have a minimal Koszul-Sullivan model   
$$
\xymatrix@=15pt{
A_{PL}(M) & A_{PL}(P) \ar[l]_{A_{PL}(i)} & \ar[l]_{A_{PL}(p)} A_{PL}(B) \\
(\wedge W, \overline{d}) \ar[u]_{\simeq}^{\overline{m}} &
 (\wedge V \otimes \wedge W, \widehat{d}) \ar[u]_{\simeq}^{\widehat{m}}
\ar@{->>}[l]_(0.57){\pi} & \ar@{>->}[l]_(0.38){j} \  (\wedge V, d)
\ar[u]_{\simeq}^{m} ,
}
$$
for ${\mathcal F}$ in which vertical arrows are quasi-isomorphisms,
$j$ is a Sullivan model for $p$ and $\pi$ is the natural projection; 
see \cite{Hal1}, \cite{Hal2}.  Moreover  the result \cite[Proposition 17.9]{F-H-T} 
allows us
to obtain a fibration ${\mathcal F}_{(\K)}$ of the form 
$$
\xymatrix@=15pt{
M_{(\K)}:=|(\wedge W, \overline{d})| \ar[r]_(0.5){|\pi|} &  
 |(\wedge V \otimes \wedge W, \widehat{d})|  \ar[r]_(0.6){|i|} &
|(\wedge V, d)|, 
}
$$ 
which is referred to as the {\it rationalized fibration} of  
${\mathcal F}$. 
Here $|A|$ denotes the spatial realization of a DGA $A$.   
Observe that $|(\wedge V, d)|$ is the rationalization $B_\K$ 
since by assumption $B$ is simply-connected.  

For a simplicial set $K$ and a DGA $A$, we define a map 
$$
\eta : \text{DGA}(A, \Omega K) \stackrel{}{\to} 
\text{Simp}(K, \Delta A)
$$
by $\eta(\phi)=f; f(\sigma)(a)= \phi(a)(\sigma)$ for $a \in A$ and
$\sigma \in K$. 
Recall that $\eta$ is bijective and that it gives rise to the adjunction 
$\psi_A : A \to \Omega \Delta A$ which is a quasi-isomorphism if $A$ is a
Sullivan algebra; see \cite[10.1 Theorem (ii)]{B-G}.
Let $\alpha : (\wedge V, d) \stackrel{\simeq}{\to} A_{PL}(X)$ be a
Sullivan model for $X$.  
The naturality of $\eta$ enables us to conclude that 
$\Omega \eta(\alpha)\circ \psi_{\wedge V}= \alpha$. 
This implies that  
$|\eta(\alpha)| : X\cong |S_*(X)| \to |(\wedge V, d)|$ 
is an isomorphism on the rational cohomology; 
see Section 2 for the notations.
%By the minimality of the model for ${\mathcal F}$, the DGA's 
%$(\wedge V \otimes \wedge W, \widehat{d})$ is a Sullivan model.  

Assume that $M$ is a ${\mathbb T}^{2k}$-separable symplectic manifold 
which is not necessarily nilpotent. 
As in the introduction, we define a map 
$p_M' : M_{(\K)} \to {\mathbb T}^{2k}_{\K}$. 
Moreover define a map $\kappa' : H^1({\mathbb T}^{2k}) \to 
H^2(B\aut{M_{(\K)}})$ by composite  
$$
\xymatrix@C15pt{
 H^1({\mathbb T}^{2k}) \ar[r]^{(p_M')^*} & H^1(M_{(\K)}) 
\ar[r]^(0.4){ev^*} & H^1(\aut{M_{(\K)}}) \ar[r]^(0.45){\tau} & 
H^2(B\aut{M_{(\K)}})
}
$$
as in Definition \ref{defn:kappa}. 
We present a variant of Theorem \ref{thm:kappa}.  

\begin{thm} 
\label{thm:kappa'}
 Let ${\mathcal F} : 
(M, \omega)  \to P
 \stackrel{p}{\to}  B$ 
be a  fibration over a simply-connected space $B$ 
with fibre $(M, \omega)$ a ${\mathbb T}^{2k}$-separable symplectic
 manifold not necessarily nilpotent. 
Let $f' : B_\K \to B\aut{M_{(\K)}}$ be the classifying map of
 the rationalized fibration  
${\mathcal F}_{(\K)} : M_{(\K)} \to  P_{(\K)} \to B_\K$ of ${\mathcal F}$.  
Then the class $[\omega] \in H^2(M; {\mathbb R})$ extends to 
a cohomology class of $P$ 
if and only if $H^*(f')\circ \kappa'$ is trivial.  
\end{thm}

\begin{proof}
Consider the commutative diagram 
$$
\xymatrix@=15pt{
M \ar[r]^j \ar[d]_{|\eta(\overline{m})|} & P \ar[r]^p 
\ar[d]_{|\eta(\widehat{m})|} & B   \ar[d]_{|\eta(m)|} \\
|(\wedge W, \overline{d})| \ar[r]_(0.4){|\pi|} &  
 |(\wedge V \otimes \wedge W, \widehat{d})|  \ar[r]_(0.6){|i|} &
|(\wedge V, d)|. 
}
$$
%Observe that vertical arrows are isomorphisms on the rational
% cohomology. This follows from the naturality of $\eta$ and the fact
% that  the adjunction $\psi_A$ is a quasi-isomorphism if $A$ is a
% Sullivan algebra.   
Observe that the low sequence is the rationalized fibration 
${\mathcal F}_{(\K)}$ of ${\mathcal F}$.  In our case, the assumption (2.9) is satisfied. 
Therefore by applying the same argument as in the proof of Theorem \ref{thm:kappa} to ${\mathcal F}_{(\K)}$, 
we have the result. 
\end{proof}

\begin{ex}
\label{ex:non-nil}An interesting example 
of a non-nilpotent ${\mathbb T}^2$-separable 
symplectic space appears  
among total spaces of ${\mathbb T}^2$-bundles over  ${\mathbb T}^2$.  
For example, we consider the ${\mathbb T}^2$-bundle 
${\mathbb T}^2 \to M(-I, I, 0, 0) \to {\mathbb T}^2$ whose total space
 is defined by 
$$
 M(-I, I, 0, 0) = {\mathbb T}^2\times {\mathbb R}^2 / \sim,  
$$ 
where 
$$
( \left(
\begin{array}{ccc}
s \\
t \\
\end{array}
\right), 
 \left(
\begin{array}{ccc}
x+1 \\
y \\
\end{array}
\right) ) \sim
( \left(
\begin{array}{ccc}
-s \\
-t \\
\end{array}
\right), 
 \left(
\begin{array}{ccc}
x \\
y \\
\end{array}
\right) ) 
$$ 
and 
$$
( \left(
\begin{array}{ccc}
s \\
t \\
\end{array}
\right), 
 \left(
\begin{array}{ccc}
x \\
y+1 \\
\end{array}
\right) )  \sim
( \left(
\begin{array}{ccc}
s \\
t \\
\end{array}
\right), 
 \left(
\begin{array}{ccc}
x \\
y \\
\end{array}
\right) ). $$
We refer the reader to \cite{S-F} 
for the details and the classification of such 
${\mathbb T}^2$-bundles. Then it follows from the argument in 
\cite[3.2]{Ke} that $M(-I, I, 0, 0)$ has a minimal model of the form 
$$
(\wedge(a, b)\otimes \wedge (\sigma, \tau), d)
$$
for which $da=db=d\sigma=0$, $d\tau = \sigma^2$, 
$\deg a = \deg b = 1$ and  $\deg \sigma = 2$. 
Thus the total space $M(-I, I, 0, 0)$ is ${\mathbb T}^2$-separable. 

We observe that $\pi_1({\mathbb T}^2)$ does not act nilpotently 
$H^*({\mathbb T}^2)$ the cohomology of the fibre of ${\mathcal F}$ and
 that $M(-I, I, 0, 0)$ admits a symplectic structure compatible with the
 fibration; see \cite[3.2]{Ke} for more details.   
Drawing on the representation of the fundamental group of the total
 space of a ${\mathbb T}^2$-bundle over  ${\mathbb T}^2$ 
due to Sakamoto and Fukuhara \cite{S-F}, we see that 
$\pi:=\pi_1(M(-I, I, 0, 0))$ is not nilpotent. In fact, there exist
 four generators of $\gamma$, $\delta$, $l$ and $h$ of 
$\pi_1(M(-I, I, 0, 0))$ such that, in particular,  
$\gamma (l, h) \gamma^{-1} = (l, h)(-I)$ and that $l$ and $h$ correspond to 
appropriate generators of the fundamental group of the fibre. 
We then have $\gamma l \gamma^{-1} = l^{-1}$ and hence 
 $\gamma l^k \gamma^{-1} l^{-k} = l^{-2k}\neq 0$ 
for any $k \in {\mathbb Z}$. Observe that the induced map 
$\pi_1({\mathbb T}) \to \pi$ is a monomorphism. 
An inductive argument deduces that 
 $\gamma l^{2^{k-2}} \gamma^{-1} l^{-(2^{k-2})}$ is in 
$[\pi, \cdots [\pi, [\pi, \pi]] \cdots ]$ ($k$-times). This implies that 
$\pi$ is not nilpotent.  
\end{ex}

It seems that Theorem \ref{thm:kappa'} is somewhat technical
and forced. However it is indeed a generalization of 
Theorem \ref{thm:kappa} as is described below.

\begin{rem} 
\label{rem:generalization}
Let ${\mathcal F} : M \to P \to B$ be a fibration over a
 simply-connected base with nilpotent fibre. 
Then $\lambda :=\eta (\overline{m}) : M \to |(\wedge V, d)|=M_{(\K)}$
 mentioned in the proof of Theorem \ref{thm:kappa'} is the rationalization 
 map; see \cite{B-G}. Thus the pullback fibration 
$M_\K \to E \to B$ by the rationalization map 
$\eta(m) : B \to B_\K$ of the rationalized fibration 
${\mathcal F}_{(\K)} : M_\K \to P_{(\K)} \to B_\K$ 
is a fibrewise localization. 
The result \cite[Theorem 3.2]{May2} enables us to obtain a canonical map 
$\lambda_\sharp : B\aut{M} \to B\aut{M_\K}$, which is an isomorphism on
 the rational cohomology. 

Let $f : B \to B\aut{M}$ be the classifying map for ${\mathcal F}$. 
Then the pullback of the universal $M_\K$-fibration 
by the composite $\lambda_{\sharp}\circ f$ is also a fibrewise localization of
 $\mathcal F$; see \cite[Theorem 4.1]{May2} and the ensuing discussion. 
The universal property of a fibrewise localization 
\cite[6.1 Theorem]{Ll} implies that 
$$
\lambda_\sharp\circ f \simeq f'\circ
 \eta(m), 
 \eqnlabel{add-4}
 $$
 where $f' : B_\K \to B\aut{M_\K}$ is the classifying map for 
 the rationalized fibration ${\mathcal F}_{(\K)}$ of ${\mathcal F}$. 
It turns out  that Theorem \ref{thm:kappa'} is a
 generalization of Theorem \ref{thm:kappa}.   
\end{rem}

\noindent
{\it Proof of Claim \ref{claim:eq}.}
 Recall the map $\lambda_\sharp : B\aut{M} \to B\aut{M_\K}$ mentioned in Remark \ref{rem:generalization}. 
Since $\lambda_\sharp$ induces an isomorphism on the rational cohomology, we consider the map 
$\lambda_\sharp$ the rationalization of $B\aut{M}$.  Observe that $B\aut{M_\K}$ is a rational space. 
Thus we have a sequence 
\begin{eqnarray*}
{\mathcal E}N{\mathcal W}(B) & \maprightud{l_*}{} &  {\mathcal E}N{\mathcal W}(B)\otimes \K 
\cong [B, B\aut{M}]\otimes \K \\
& \maprightud{(\lambda_\sharp)_*}{\cong}& [B, B\aut{M_\K}] \cong {\mathcal E}N_\K{\mathcal W}(B_\K), 
\end{eqnarray*}
where $(\lambda_\sharp)_*$ denotes the isomorphism induced by the rationalization 
$\lambda_\sharp$. 

Let ${\mathcal F}$ and ${\mathcal F}'$ be fibtarions over $B$ with fibre $N$ classified by maps $f$ and $g$, 
respectively.  Suppose that $l_*([{\mathcal F}])= l_*([{\mathcal F}'])$. 
Then $\lambda_\sharp\circ f \sim \lambda_\sharp \circ g$. 
It follows from (5.1) that $f'\circ \eta(m) \sim g'\circ \eta(m)$, where $f'$ and $g'$ are classifying maps 
from $B_\K$ to  $B\aut{N_\K}$ for $F_{(\K)}$ and $F_{(\K)}'$, respectively. 
The rationalization map $\eta(m) :  B \to B_\K$ induces a bijection 
$\eta(m)^* : [B_\K, B\aut{N_\K}] \to [B, B\aut{N_\K}]$. This allows us to conclude that  $f'\sim g'$. The converse also follows from (5.1). We have the result.   
\hfill\qed

\end{document}